\documentclass[smallcondensed]{svjour3}     

\smartqed  
\usepackage{colortbl}
\usepackage{longtable}
\usepackage{hyperref}
\usepackage{amsmath}
\usepackage{pslatex}
\usepackage{amssymb}
\usepackage{latexsym}
\usepackage{graphicx}
\usepackage{multicol}
\usepackage{url}
\usepackage{ulem}
\usepackage{afterpage}
\begin{document}

\title{Stability of the Rhomboidal Symmetric-Mass Orbit}

\author{Lennard Bakker \and Skyler Simmons}


\institute{Lennard Bakker \at
              275 TMCB \\
              Brigham Young University \\
              Provo, UT, 84602 \\
              \email{bakker@math.byu.edu}           
           \and
           Skyler Simmons \at
              275 TMCB \\
              Brigham Young University \\
              Provo, UT, 84602 \\
              \email{xinkaisen@gmail.com}           
}

\date{Received: date / Accepted: date}

\maketitle

\begin{abstract}
We study the rhomboidal symmetric-mass 4-body problem in both a two-degree-of-freedom and a four-degree-of-freedom setting.  Under suitable changes of variables in both settings, isolated binary collisions at the origin are regularizable.  Linear stability analysis is performed in both settings.  For the two-degree-of-freedom setting, linear stability is established for a wide interval of mass ratios.  A Poincar\'{e} section analysis is also performed, showing stability.  In the four-degree-of-freedom setting, linear stability fails except for a very small interval of mass ratios.
\end{abstract}

\section{Introduction}

In the \textit{Principia Mathematica}, published in 1687, Newton outlined many governing principles of the motion of physical objects.  Combining the laws $F = ma$ and the law of universal gravitation gave a relation that could describe the motions of bodies in space.  The resulting equations helped to explain many of the behaviors that astronomers of the time were aware of (most notably Kepler).  Mathematically, the study of determining the motion of $n$ point masses in space is known as the Newtonian $n$-body problem.  Notationally, if $\{q_1, q_2, ..., q_n\}$ represent the positions of the bodies in $\mathbb{R}^k$ ($k = 1, 2,$ or $3$) with masses $\{m_1, m_2, ..., m_n\}$ respectively, then their motion is governed by the system of differential equations
\begin{equation}
m_i \ddot{q}_i  = \sum_{i \neq j} \frac{m_i m_j (q_j - q_i)}{|q_i - q_j|^3},
\label{standardmotion}
\end{equation}
where the dot represents the derivative with respect to time.  Despite hundreds of years of study and the relatively recent development of computer ODE solvers, many open questions about the $n$-body problem remain.  \\

One aspect of the $n$-body problem that has been getting much attention of late are orbits involving collision singularities.  A \textit{collision singularity} occurs when $q_i = q_j$ for some $i \neq j$.  In the equations governing motion, this results in a zero denominator in one or more terms in the sum.  Under certain conditions, these collisions can be \textit{regularized} and continued past collision. \\

Schubart \cite{bibSchubart} was one of the first to study periodic orbits with collisions.  He was able to find a collinear three-body equal-mass orbit where the central body alternated between collisions with the outer two.  This was further extended to the case of arbitrary masses numerically by H\'enon \cite{bibHenon} in 1977.  Analytic existence of the equal-outer-mass orbit was established independently by Venturelli \cite{bibVenturelli} and Mockel \cite{bibMoeckel}, both in 2008.  Shibiyama \cite{bibShib1} recently demonstrated the existence of the arbitrary-mass version.  The study of linear stability of Schubart's orbit was performed by Hietarinta and Mikkola \cite{bibHM1} in 1993.\\

Sweatman found a Schubart-like collinear four-body symmetric orbit in 2002 \cite{bibSweatman1}, and later studied its linear stability \cite{bibSweatman2}.  This orbit features simultaneous binary collisions between two outer pairs of bodies followed by an interior collision between the two central bodies.  Analytic existence of this orbit was given by Ouyang and Yan in \cite{bibYan3}.  \\

Planar orbits with singularities have also been studied.  A planar four-body orbit featuring simultaneous binary collisions was described in \cite{bibOYS}.  The orbit was shown to be linearly stable in \cite{bibBORSY}.  It was later shown that this orbit could be extended to symmetric masses in \cite{bibBOYS1} (see also \cite{bibBOYS2}), and linear stability for this extension was shown for an interval of certain mass ratios in \cite{bibBMS}.  \\

Analytic existence of large families of orbits with singularities was recently proven by Shibayama in \cite{bibShib1} and Martinez in \cite{bibMartinez}.  Each orbit can be reduced to two position and two momentum variables (the so-called two-degree-of-freedom problem).  One orbit of note in this family is the rhomboidal four-body orbit, which features two pairs of bodies on the $x$- and $y$-axes.  The pairs collide at the origin in an alternating fashion.  This orbit was shown to exist analytically in multiple independent papers (by Yan in \cite{bibYan1} and Martinez in \cite{bibMartinez} for equal masses, \cite{bibShib1} for symmetric masses).  Additionally, Yan showed that for equal masses, the orbit is linearly stable.\\

In a separate study of the rhomboidal four-body problem with unequal masses, Waldvogel \cite{bibWaldvogel1} notes that ``sufficiently simple systems may bear the chance of permitting theoretical advances,'' and identifies the rhomboidal configuration as one such system.  Indeed, much of the analysis performed in this paper is far simpler than that of \cite{bibBMS}.  \\

The remainder of this paper is divided into two principal sections.  Section \ref{sec2df} concerns the orbit in the two-degree-of-freedom (2DF) setting.  In Section \ref{sec2dper} we give the notation and mathematical description of the orbit.  Section \ref{LinStab1} outlines some basic theory for linear stability.  Section \ref{Numerics} describes some preliminary numerical calculations that are needed to study the orbit.  Section \ref{Poincare} gives a Poincar\'{e} section analysis of the orbit in the 2DF setting. \\

In Section \ref{sec4df} we further the study of the orbit in the four-degree-of-freedom (4DF) setting.  Section \ref{sec4dfper} sets up the additional mathematical notation for the new setting.  Section \ref{syms} describes the symmetries of the orbit, which will be needed for the stability calculations.  Section \ref{stabcalc} reviews how symmetries of a periodic orbit can be used to simplify the linear stability calculation.  Section \ref{apriori} shows how the stability calculation can be reduced to the calculation of three entries of a particular matrix.  Section \ref{Results} gives the results and implications of the remaining calculation. \\

\section{The Rhomboidal Two-Degree-of-Freedom Symmetric-Mass Problem}\label{sec2df}

\subsection{The Periodic Orbit}\label{sec2dper}

We consider the planar Newtonian $4$-body problem with bodies located at 
\begin{equation}
q_1 = (x_1, 0)\text{, }q_2 = (0,x_2)\text{, } q_3 = -q_1\text{, } q_4 = -q_2
\label{config}
\end{equation}
and masses $1$, $m$, $1$, $m$ respectively for some $m \in (0,1]$.  The bodies travel along the $x$ and $y$ axes, forming the vertices of a rhombus at all times away from collision.  Binary collisions occur between the bodies with equal masses at the origin.  For the periodic orbit, the non-colliding bodies have zero momentum at collision time.  Following collision, the colliding bodies eject along the appropriate coordinate axis, and the remaining two bodies travel toward collision at the origin, where a similar (zero momentum of non-colliding bodies) behavior occurs.  (See Figure \ref{figTheOrbit}.)\\

\begin{figure}
\begin{center}
\includegraphics[scale=.5]{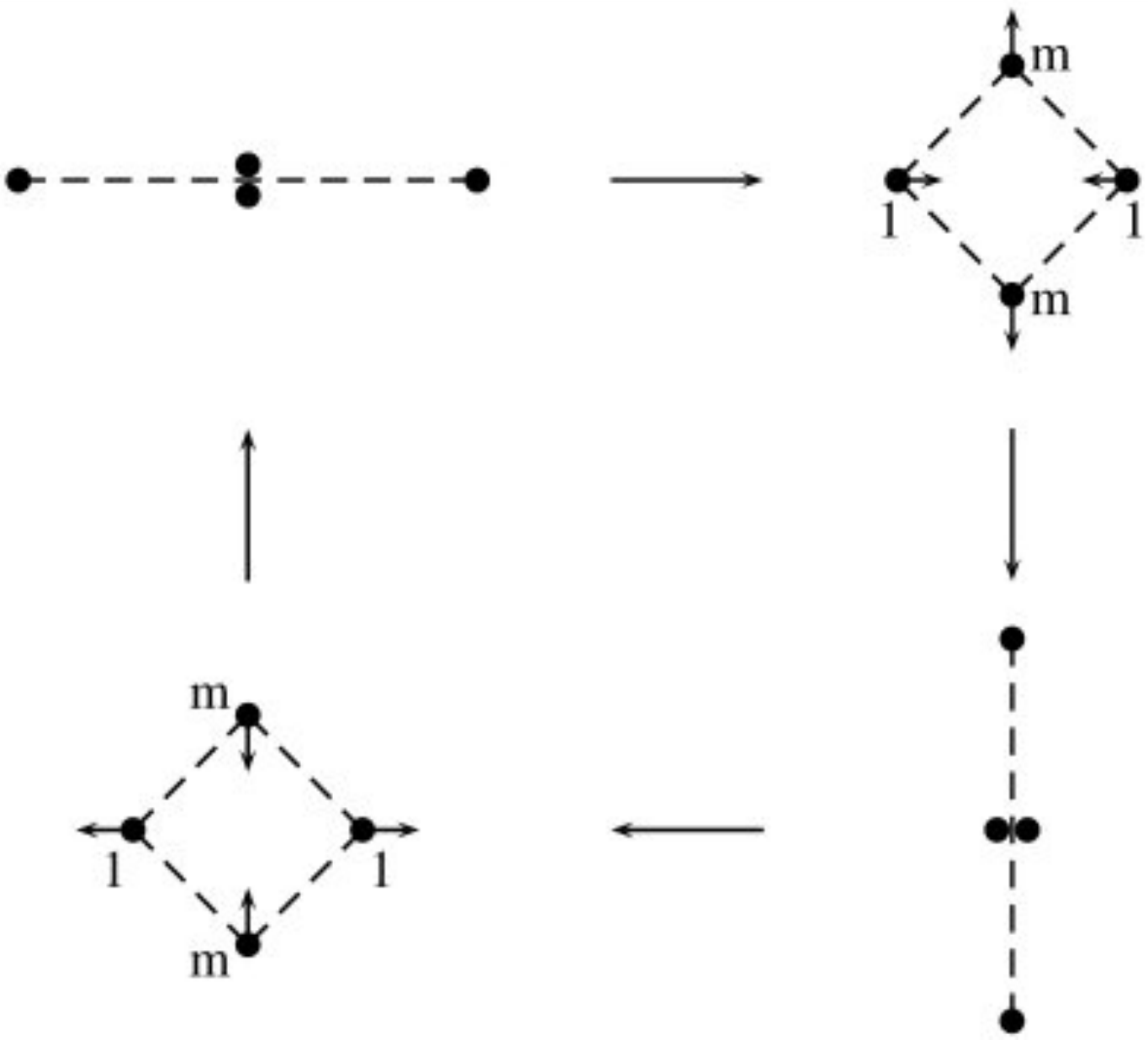}
\caption{The rhomboidal four-body orbit.}
\label{figTheOrbit}
\end{center}
\end{figure}

Analytic existence of this orbit was shown in \cite{bibYan1} in the $m = 1$ case, and for $m \in (0,1]$ in \cite{bibShib1}.  The Hamiltonian for the system is given by
$$H = \frac{1}{4}w_1^2 + \frac{1}{4m}w_2^2- \frac{1}{2x_1} - \frac{m^2}{2x_2} - \frac{4m}{\sqrt{x_1^2 + x_2^2}}$$
where $w_1 = 2\dot{x}_1, w_2 = 2m\dot{x}_2$.  We can continue the orbits past collision via a regularization under which binary collision corresponds to an elastic bounce.  To regularize these collisions, we use a Levi-Civita-type change of coordinates.  Using the canonical transformations $Q_i^2 = x_i$, $P_i = 2Q_i w_i$ for $i = 1, 2$, with a change of time satisfying $dt/ds = x_1x_2$, the regularized Hamiltonian in extended phase space is given by
\begin{align*}
\Gamma &= \Gamma(Q_1, Q_2, P_1, P_2, E) \\
&=\frac{dt}{ds}(H-E) \\
&=\frac{1}{16}Q_2^2 P_1^2 + \frac{1}{16}\frac{Q_1^2 P_2^2}{m} - \frac{1}{2}Q_1^2m^2 - \frac{1}{2}Q_2^2 - \frac{4 Q_1^2 Q_2^2 m}{\sqrt{Q_1^4 + Q_2^4}} - Q_1^2 Q_2^2 E.
\end{align*}
This yields the equations of motion
\begin{align}
Q_1' &= \frac{1}{8}Q_2^2 P_1, \label{regMotion1}\\
Q_2' &= \frac{1}{8m} Q_1^2 P_2
\end{align}
\begin{align}
P_1' &= -\frac{1}{8m}Q_1 P_2^2 + Q_1 m^2 + \frac{8Q_1 Q_2^2 m}{(Q_1^4 + Q_2^4)^{1/2}} - \frac{8Q_1^5 Q_2^2 m}{(Q_1^4 + Q_2^4)^{3/2}} + 2Q_1Q_2^2 E, \\
P_2' &= -\frac{1}{8}Q_2 P_1^2 + Q_2 + \frac{8Q_1^2 Q_2 m}{(Q_1^4 + Q_2^4)^{1/2}} - \frac{8Q_1^2 Q_2^5 m}{(Q_1^4 + Q_2^4)^{3/2}} + 2Q_1^2 Q_2 E, \label{regMotion4}
\end{align}
where $'$ denotes the derivative with respect to the new time variable $s$. \\

At the time of collision of the two bodies on the $x$-axis, we have $Q_1 = 0$ and $P_2 = 0$.  At this time, setting $\Gamma = 0$ yields
$$Q_2^2 \left(\frac{1}{16}P_1^2 - \frac{1}{2} \right) = 0.$$
Hence, $P_1 = \pm 8^{3/2}$, with the sign being the same as the sign on $Q_2$.  Similarly, at the time of collision of the two bodies on the $y$-axis, $Q_2 = P_1 = 0$.  The condition $\Gamma = 0$ then gives
$$Q_1^2 \left(\frac{1}{16m}P_2^2 - \frac{1}{2}m^2 \right) = 0,$$
and so $P_2 = \pm (8m)^{3/2}.$ \\

Let
\begin{equation}
S = 
\begin{bmatrix}
-1 & 0 & 0 & 0 \\
0 & 1 & 0 & 0 \\
0 & 0 & 1 & 0 \\
0 & 0 & 0 & -1
\end{bmatrix}.
\label{SMatrix2df}
\end{equation}
A standard proof shows that if $\gamma(s)$ is a $T$-periodic solution to (\ref{regMotion1}) - (\ref{regMotion4}), both $-S\gamma(T/2 - s)$ and $S\gamma(T - s)$ are solutions as well.  Existence and uniqueness of solutions then imply that
$$-S\gamma(T/2 - s) = \gamma(s) = S\gamma(T-s)$$
for all $s$.  Hence the symmetry group for the rhomboidal four-body orbit is isomorphic to the Klein four group, with $S$ and $-S$ as generators.

\subsection{Linear Stability} \label{LinStab1}
Note that $\Gamma$ is a smooth function defined on $\mathbb{R}^{4} \setminus \{Q_1 = Q_2 = 0\}$.  Suppose $\gamma(s)$ is a $T$-periodic solution of the system $z' = JD\Gamma(z)$, where $' = d/ds$,
\begin{equation*}
J = 
\begin{bmatrix}
O & I \\
-I & O \\
\end{bmatrix},
\end{equation*}
and $I$ and $O$ are the $2 \times 2$ identity and zero matrices, respectively.  If $X(s)$ is the fundamental matrix solution of the linearized equations

\begin{equation}
\xi' = JD^2(\gamma(s))\xi, \quad \xi(0) = I
\label{linearized1}
\end{equation}
then the monodromy matrix is given by $X(T)$ and satisfies $X(s + T) = X(s)X(T)$ for all $s$.  Eigenvalues of the monodromy matrix are also the characteristic multipliers of $\gamma$, and therefore determine the linear stability of $\gamma$.  In particular, $\gamma$ is spectrally stable if all of its characteristic multipliers lie on the unit circle, and $\gamma$ is linearly stable if it is spectrally stable and semisimple apart from trivial eigenvalues. \\

Linear stability is typically established by numerical integration.  Some elegant techniques for simplifying the numerical work were presented by Roberts in \cite{bibRoberts1}, and will be presented in Section \ref{stabcalc}. \\

\subsection{Numerical Determination of Initial Conditions}\label{Numerics}

In order to determine the initial conditions for the rhomboidal orbit, we model each of $Q_1, Q_2, P_1$, and $P_2$ by truncated trigonometric polynomials:
\begin{align}
\tilde{Q}_1 &= \sum_{i = 0}^n a_i \sin((2i+1)s), \label{trigFirst} \\
\tilde{Q}_2 &= \sum_{i = 0}^n b_i \sin((2i+1)(s+\pi/2)), \\
\tilde{P}_1 &= \sum_{i = 0}^n c_i \sin((2i+1)(s-\pi/2)), \\
\tilde{P}_2 &= \sum_{i = 0}^n d_i \sin((2i+1)s). \label{trigLast}
\end{align}
The choice of trigonometric polynomials is natural for modeling periodic behavior.  A similar technique was carried out by Sim\'o in \cite{bibSimo1}.  The time shifts and choice of odd-only multiples of $s$ correspond to symmetries of the orbit.  In particular, for these polynomials, the time-reversing symmetries shown earlier are built-in, and the non-colliding bodies have zero net momentum at collision time.  For a fixed $n$, we numerically minimize the value of
$$\int_0^{2\pi} \left((Q_1' - \tilde{Q}'_1)^2 + (Q_2' - \tilde{Q}'_2)^2 + (P_1' - \tilde{P}'_1)^2 + (P_2' - \tilde{P}'_2)^2\right) \ ds$$
where the minimization is taken over the space of coefficients $\{a_i, b_i, c_i, d_i\}$.  We combine this with a root-finding technique to find the appropriate value of $E$ for a $2\pi$-periodic orbit, as in \cite{bibBOYS1} (see also \cite{bibBOYS2}).  Once these trigonometric polynomials are determined, we can extract the initial conditions for the orbit by evaluating $\tilde{Q}_1, \tilde{Q}_2, \tilde{P}_1, \tilde{P}_2$ at any fixed time $s \in [0,2\pi]$. \\

We obtain the initial conditions for the $2\pi$-periodic orbit for $m = 1$ by rescaling the conditions given in \cite{bibYan1}.  It is easy to check that if $\gamma(s) = (Q_1(s), Q_2(s), P_1(s), P_2(s), E)$ is a solution for (\ref{regMotion1}) - (\ref{regMotion4}), then the solution $\gamma_\epsilon(s) = (\epsilon Q_1(\epsilon s), \epsilon Q_2(\epsilon s), P_1(\epsilon s), P_2(\epsilon s), E/\epsilon^2)$ also satisfies (\ref{regMotion1}) - (\ref{regMotion4}).  Moreover, this rescaling does not change the linear stability of the orbit.  Given the initial conditions in \cite{bibYan1} and rescaling with $\epsilon \approx 1.55$, the orbit is roughly $2\pi$-periodic, which we verify by integration using the standard fourth-order Runge-Kutta-Fehlberg algorithm.  A standard curve-fitting technique can then be used to give the coefficients $\{a_i, b_i, c_i, d_i\}$ in the equations (\ref{trigFirst}) - (\ref{trigLast}).  After that, a gradual ``step-down'' technique (as in \cite{bibBOYS1}) can be used to find the initial conditions for other values of $m \in (0,1]$.  If we assume the initial conditions occur at the time the two bodies with mass 1 collide, we have $Q_1 = P_2 = 0$.  Also, as before, we know that $P_1 = -2\sqrt{2}$, so we need only to find the values of $Q_2$ and $E$.  The results of the numerical calculation are shown in Figures \ref{q2plot} - \ref{physicalplots}. \\

\begin{figure}[h]
\begin{center}
\begin{tabular}{cc}
\includegraphics[scale=.42]{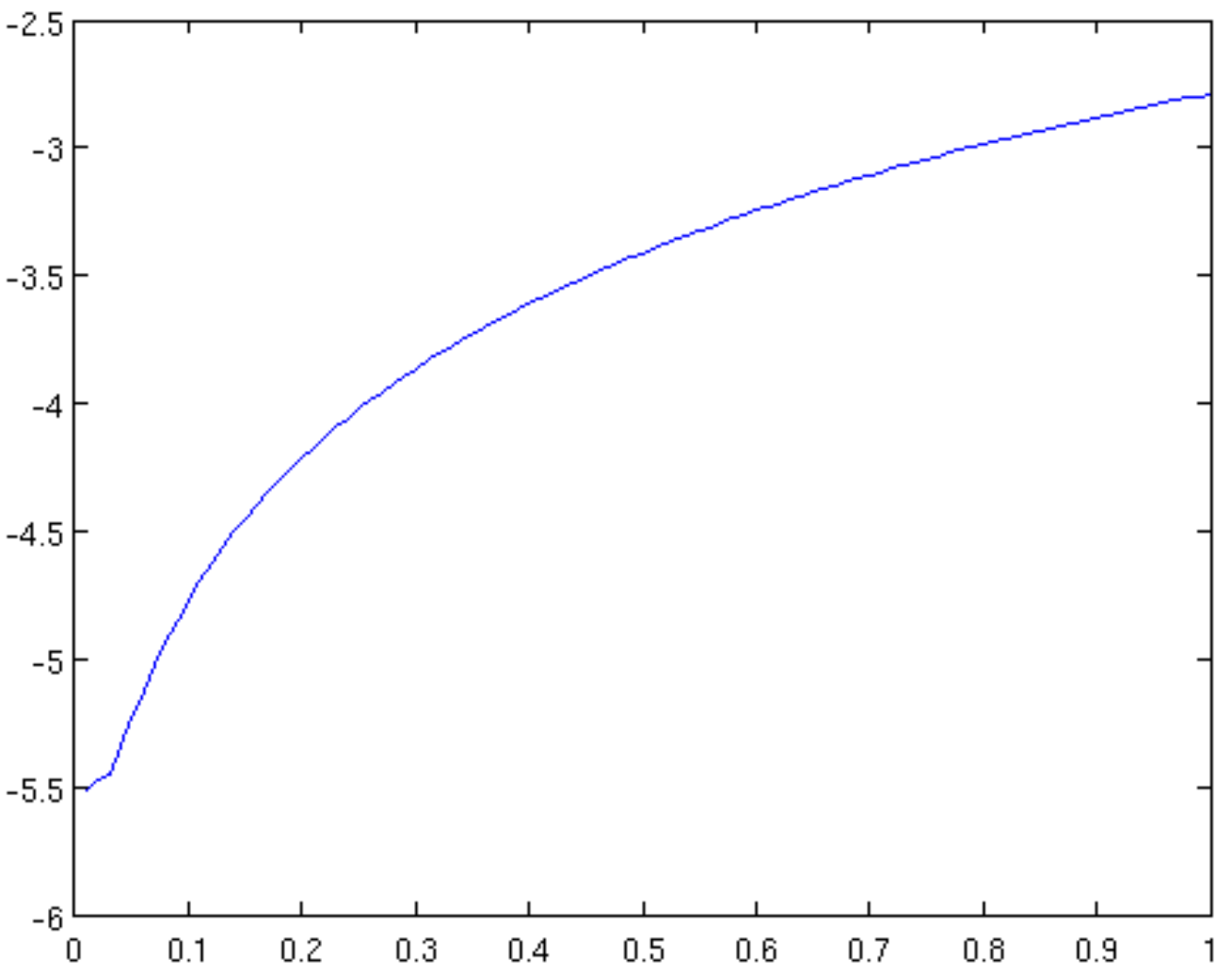} & 
\includegraphics[scale=.42]{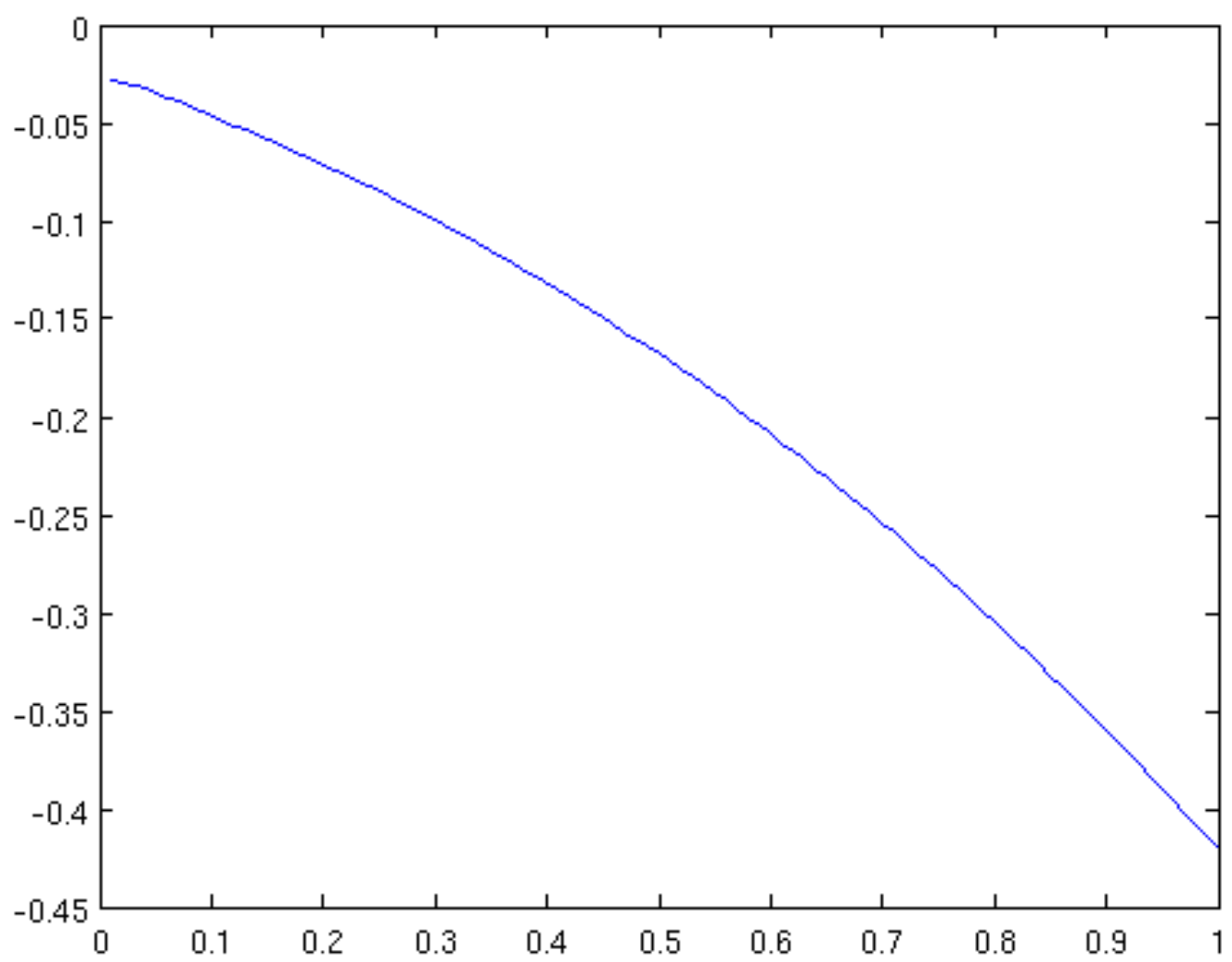} \\
\end{tabular}
\caption{The values of $Q_2$ (left) and $E$ (right) as functions of $m$.}
\label{q2plot}
\end{center}
\end{figure}

\afterpage{\clearpage
\begin{figure}[h]
\begin{center}
\includegraphics[scale=.55]{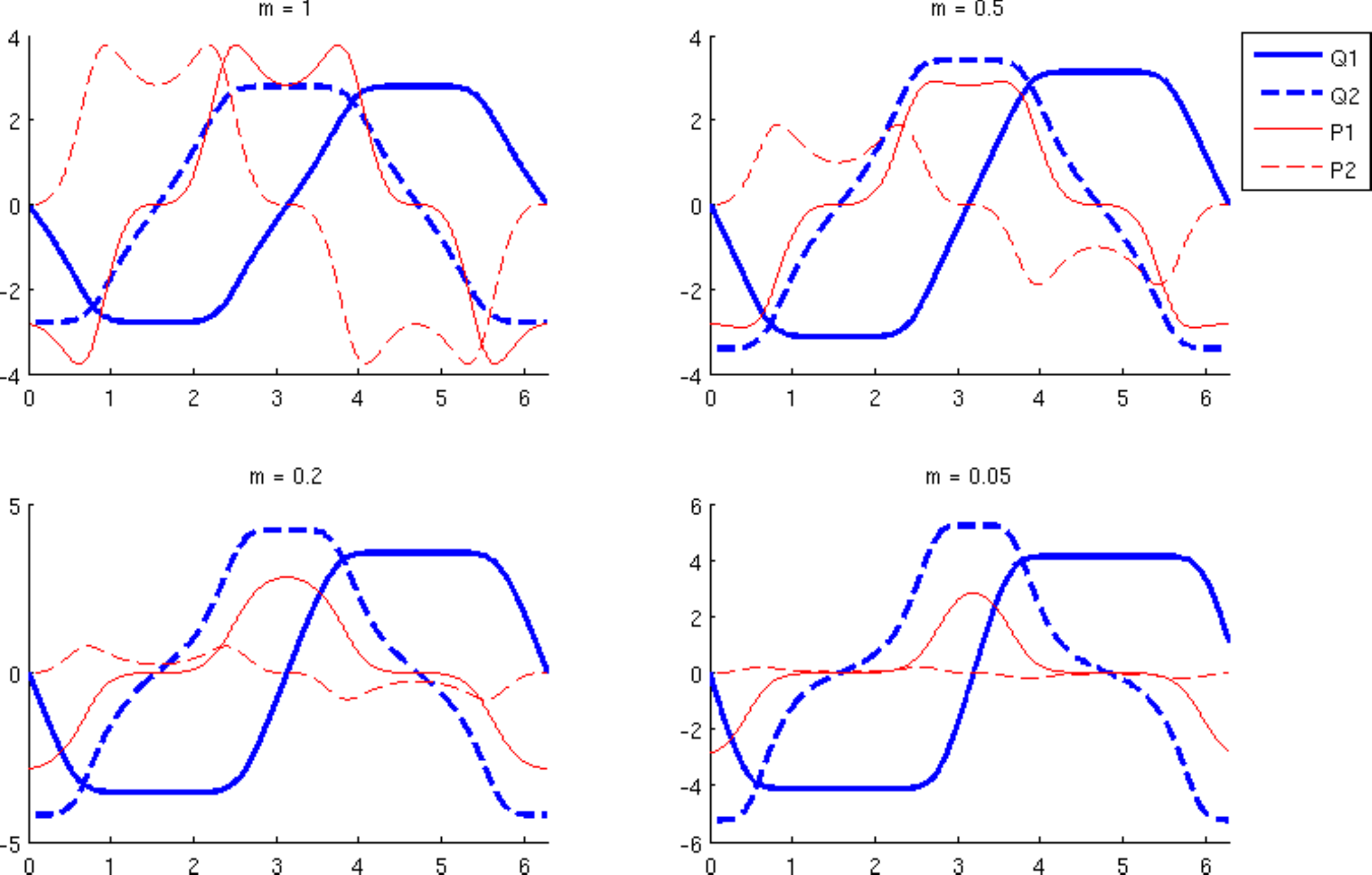}
\caption{Numerical integrations of (\ref{regMotion1}) - (\ref{regMotion4}) with initial conditions obtained from the trigonometric polynomials $\tilde{Q}_i, \tilde{P}_i$ for various masses.}
\label{regularizedplots} 
\includegraphics[scale=.55]{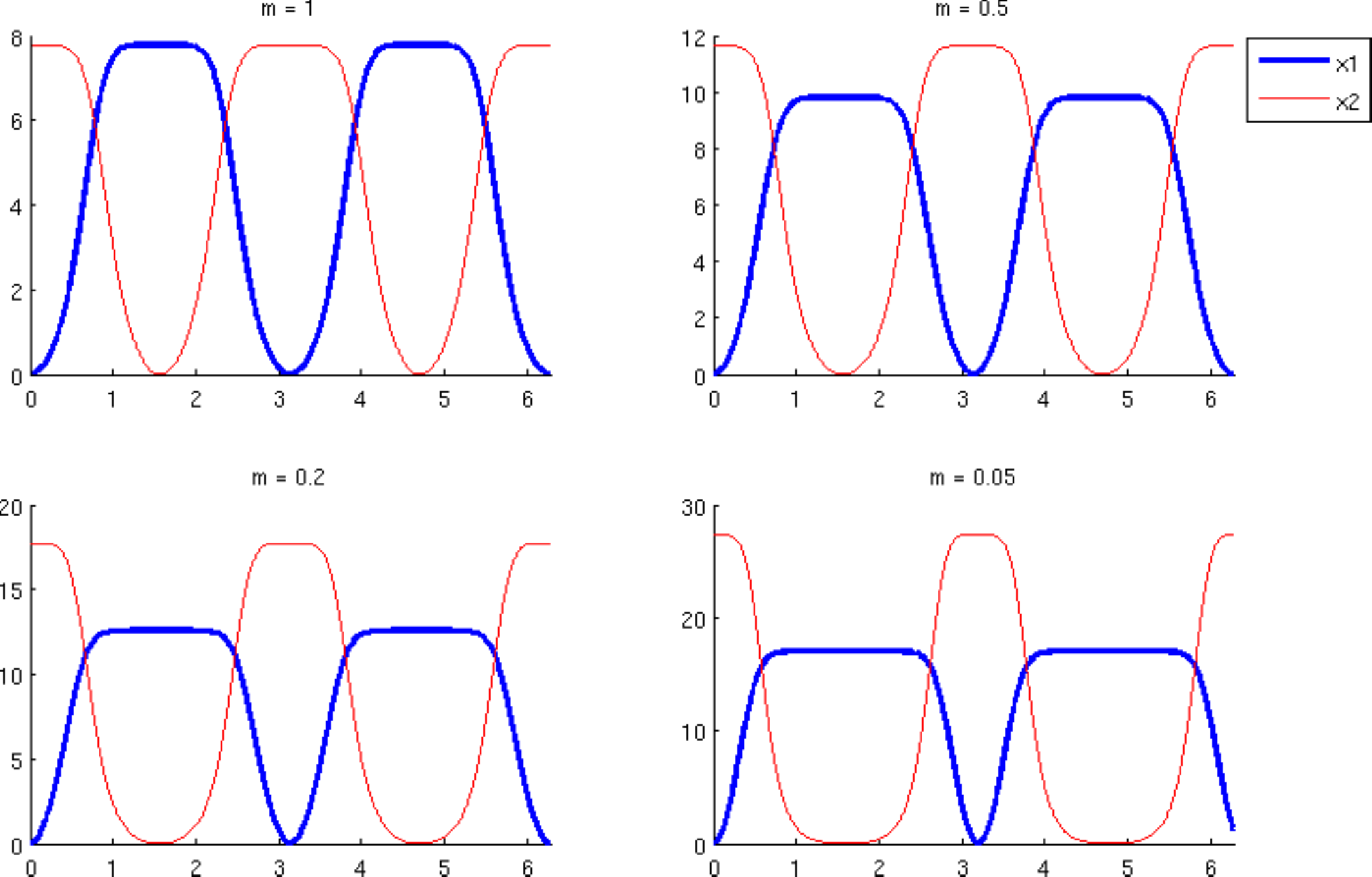}
\caption{Orbits in terms of the original physical variables $x_1$ and $x_2$ obtained by coordinate transformation of the data shown in Figure \ref{regularizedplots}.} 
\label{physicalplots}
\end{center}
\end{figure}
}

\afterpage{\clearpage}

As a result of having determined the initial conditions for the orbit, we can perform a numerical integration to determine the linear stability.  An elegant decomposition will yield the following stability result for the 2DF setting as a corollary to the stability in the 4DF setting.  As such, we postpone the proof of the following until Section \ref{Results}.
\begin{theorem}
There exists some positive number $\epsilon$ such that the 2DF symmetric-mass periodic orbit of the regularized planar rhomboidal four-body problem is linearly stable for $m \in (.01 + \epsilon, 1]$. 
\label{resulttheorem}
\end{theorem}

\subsection{Poincar\'e Section Analysis}\label{Poincare}

To numerically analyze nonlinear stability, we find a suitable Poincar\'e section for the orbit.  This was done in the $m = 1$ case in \cite{bibWaldvogel1}.  Our more general Poincar\'e section is based on techniques presented in \cite{bibHM1} and \cite{bibSweatman1}.  For any value of $m$, we seek a number $\alpha$ such that
$$\frac{x_1}{x_2} = \alpha$$
is maintained throughout the orbit, with $x_1$ and $x_2$ as defined in (\ref{config}) earlier.  In other words, the value of $\alpha$ corresponds to the ratio of $x_1$ and $x_2$ in a homographic orbit where the trajectories of the bodies correspond to total collapse (or ejection from total collapse).  We find this value of $\alpha$ by solving the standard equations of motion (\ref{standardmotion}) with the substitutions $x_1 = \alpha x_2$ and $\ddot{x_1} = \alpha\ddot{x_2}$.  Doing so, we find that the required value of $\alpha$ for a given mass $m$ is a root of the 12th-degree polynomial
\begin{equation}\label{alphaeq}
(1+\alpha^2)^3(m\alpha^3-1)^2 - 64\alpha^6(1-m)^2 = 0.
\end{equation}
Notice that if the ratio $x_1 / x_2$ is constant throughout the orbit, then the ratio $x_2 / x_1$ is also constant throughout.  It can be verified by numerical integration that the roots of \ref{alphaeq} corresponding to the ratio $x_1 / x_2$ lie in the interval $[0,1]$.  This will be preferred for ease of numerical calculation.  The value of $\alpha$ as a function of $m$ is plotted in Figure \ref{alphafigure}. \\

\begin{figure}[h]
\begin{center}
\includegraphics[scale=.5]{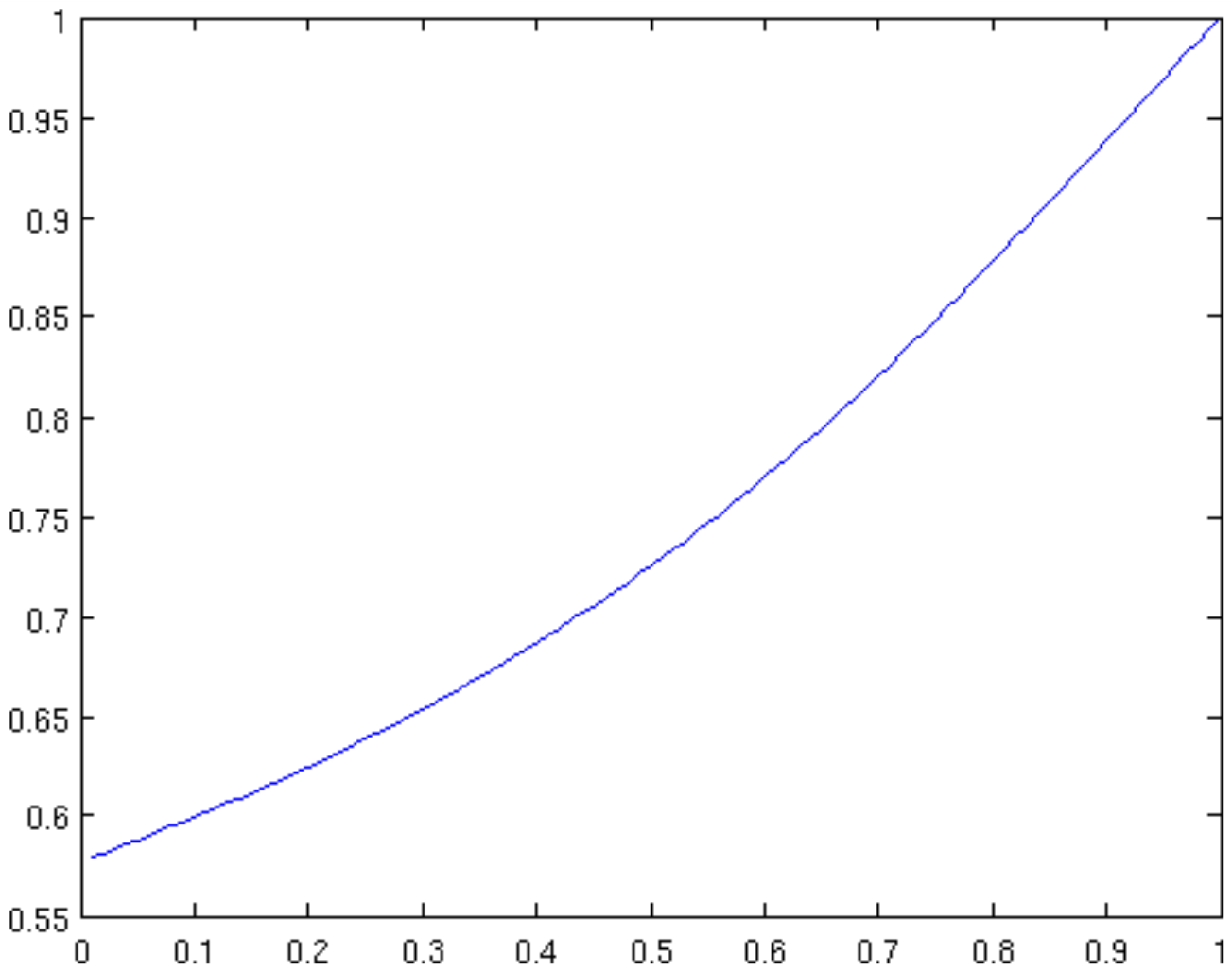}
\caption{The value of $\alpha$ as a function of $m$.}
\label{alphafigure}
\end{center}
\end{figure}

For fixed $E = -1$, we define a Poincar\'e section $\Sigma$ to be the two-dimensional surface given by $x_1 = \alpha x_2$ in the phase space defined by the variables $x_1$, $x_2$, $\dot{x}_1$, and $\dot{x}_2$.  (Note that $\dot{x}_1$ and $\dot{x}_2$ are simply linear re-scalings of $w_1$ and $w_2$.)  Restricting to $E = -1$, we find a bound on the possible values of $x_1$.  Specifically, if $\dot{x_1} = \dot{x_2} = 0$ on $\Sigma$, the condition $E = -1$ requires that
$${x_1} = \frac{1}{2} + \frac{m^2\alpha}{2} + \frac{4m}{\sqrt{1 + \frac{1}{\alpha^2}}} = r_\text{max}.$$
For a set of initial conditions on $\Sigma$, the requirement $E = -1$ necessarily implies that $x_1 \leq r_\text{max}$, and if either of $\dot{x}_1$ or $\dot{x}_2$ are non-zero, then the strict inequality $x_1 < r_\text{max}$ holds.  \\

We define coordinates $(r, \theta)$ on $\Sigma$ by
$$r = \frac{x_1}{r_\text{max}}, \quad \theta = \tan^{-1} \left(\frac{\dot{x}_1}{\alpha \dot{x}_2}\right).$$
Under this change of coordinates, the homographic orbit corresponds to the line $\theta = \pi/4$.  For a $9 \times 15$ grid of equally spaced initial conditions in $(r, \theta)$ we numerically integrate (\ref{regMotion1}) - (\ref{regMotion4}) for the corresponding initial conditions and record the first 200 intersections of the orbit with $\Sigma$.  (Integration was preemptively terminated if any of $Q_i, P_i$ exceeded 1000 in absolute value.)  The results of this are shown in Figures \ref{manypoin1} - \ref{manypoin5}.  The observed concentric rings numerically match the predicted result of Moser's Invariant Curve Theorem in \cite{bibCMLectures}, and show that the rhomboidal symmetric-mass orbit is nonlinearly stable for $m \in (.01+ \epsilon, 1]$ for the same $\epsilon$ as in Theorem \ref{resulttheorem}.

\begin{figure}[h]
\begin{center}
\includegraphics[scale=.5]{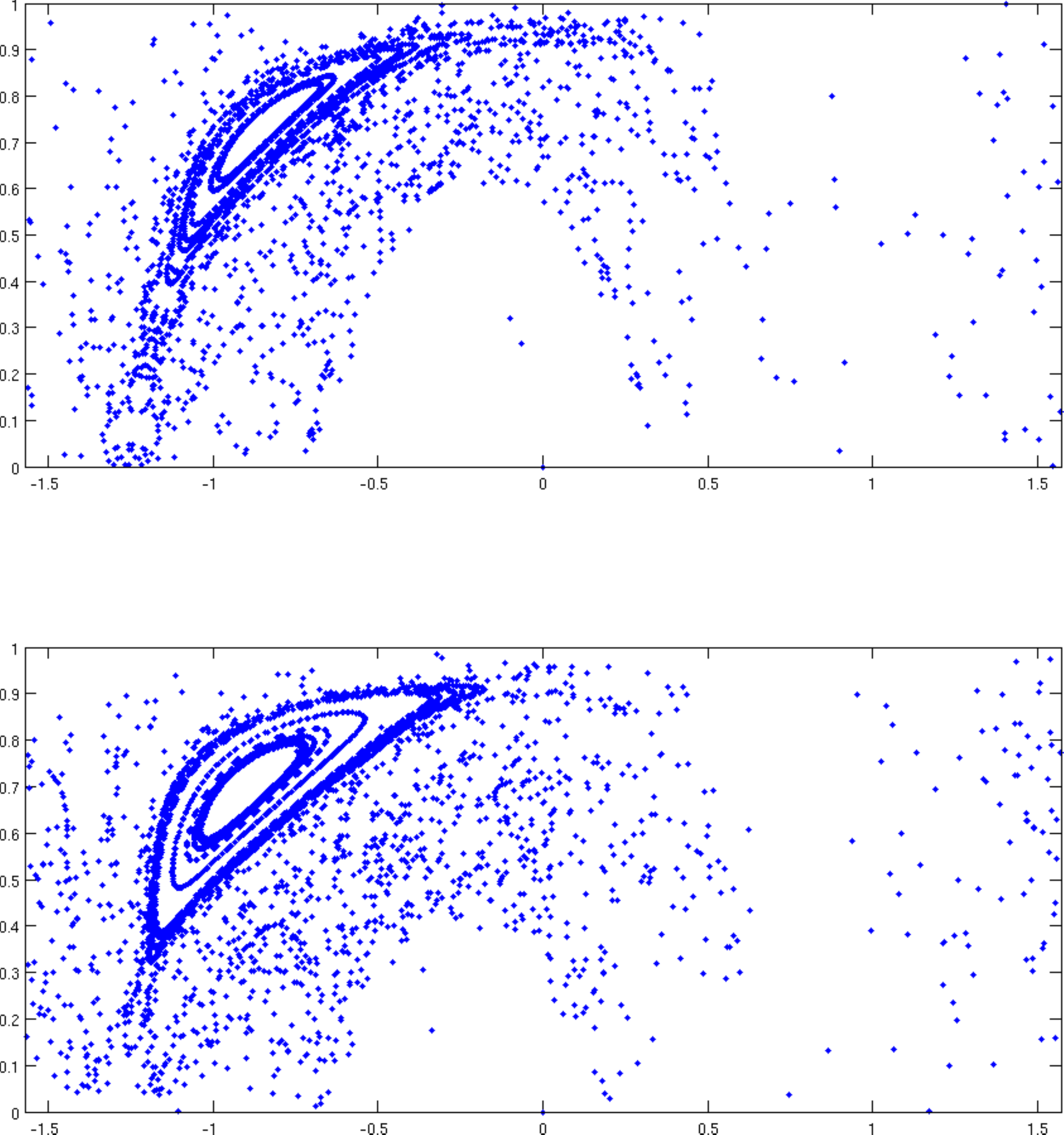}
\caption{Poincar\'e sections plotted for $m = .1$ (top) and $m = .2$ (bottom).  In these plots, $r$ lies on the vertical axis.  The homographic orbit at $\theta = \pi/4$ is not plotted for clarity.}
\label{manypoin1}
\end{center}
\end{figure}

\afterpage{\clearpage
\begin{figure}[h]
\begin{center}
\includegraphics[scale=.5]{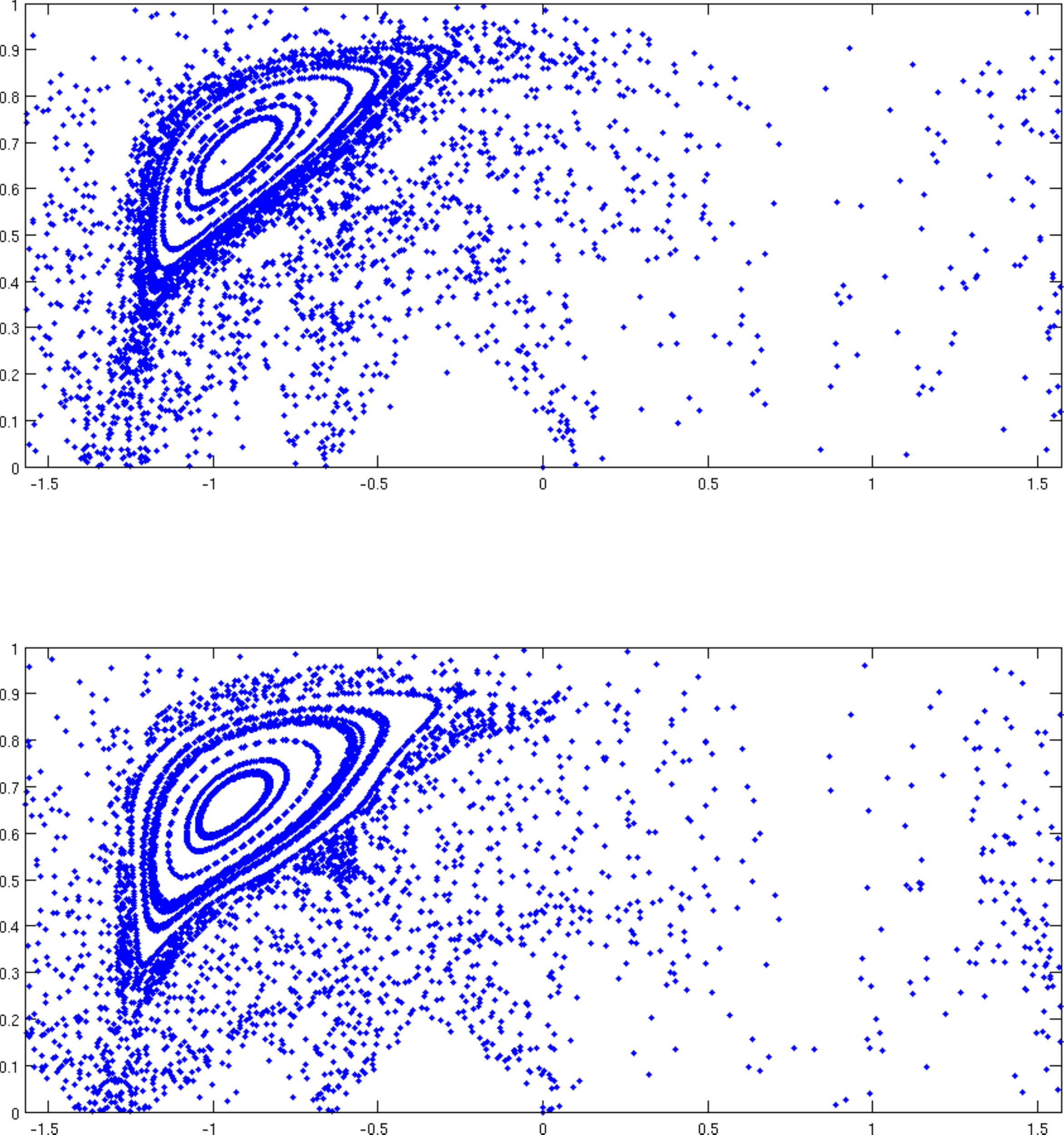}
\caption{Poincar\'e sections plotted for $m = .3$ (top) and $m = .4$ (bottom).}
\label{manypoin2}
\end{center}
\end{figure}
}

\afterpage{\clearpage
\begin{figure}[h]
\begin{center}
\includegraphics[scale=.5]{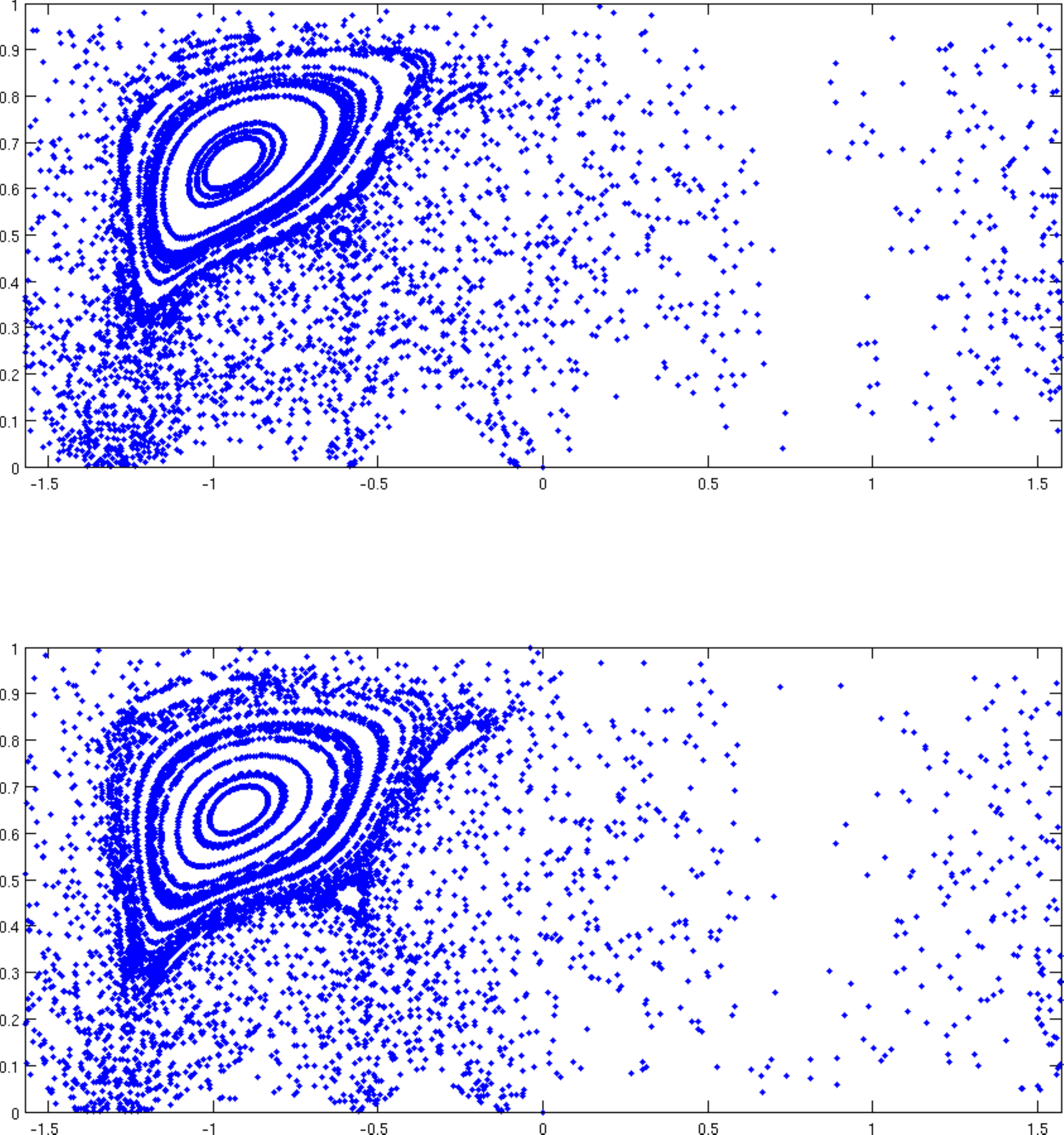}
\caption{Poincar\'e sections plotted for $m = .5$ (top) and $m = .6$ (bottom).}
\label{manypoin3}
\end{center}
\end{figure}
}

\afterpage{\clearpage
\begin{figure}[h]
\begin{center}
\includegraphics[scale=.5]{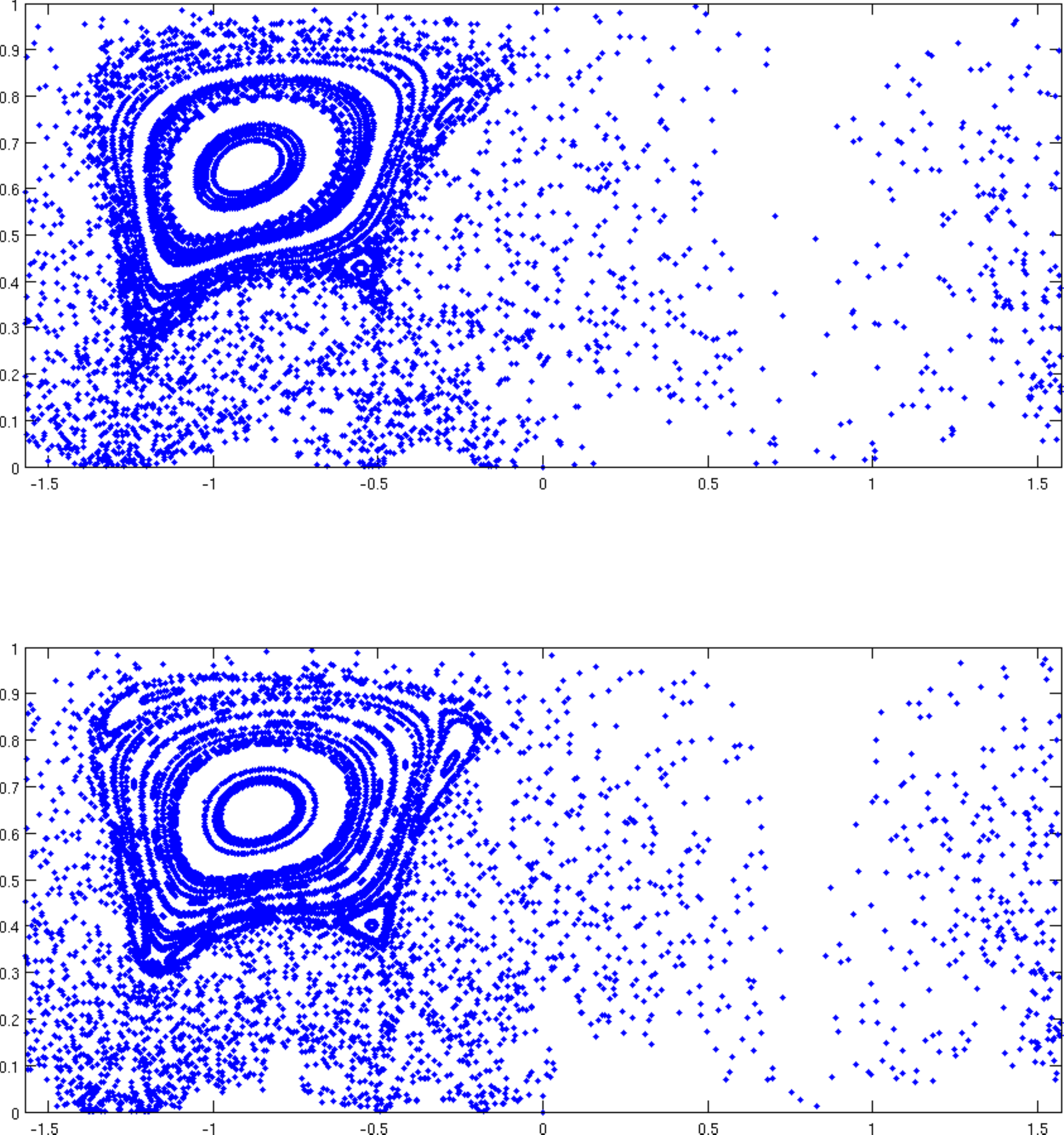}
\caption{Poincar\'e sections plotted for $m = .7$ (top) and $m = .8$ (bottom).}
\label{manypoin4}
\end{center}
\end{figure}
}

\afterpage{\clearpage
\begin{figure}[h]
\begin{center}
\includegraphics[scale=.5]{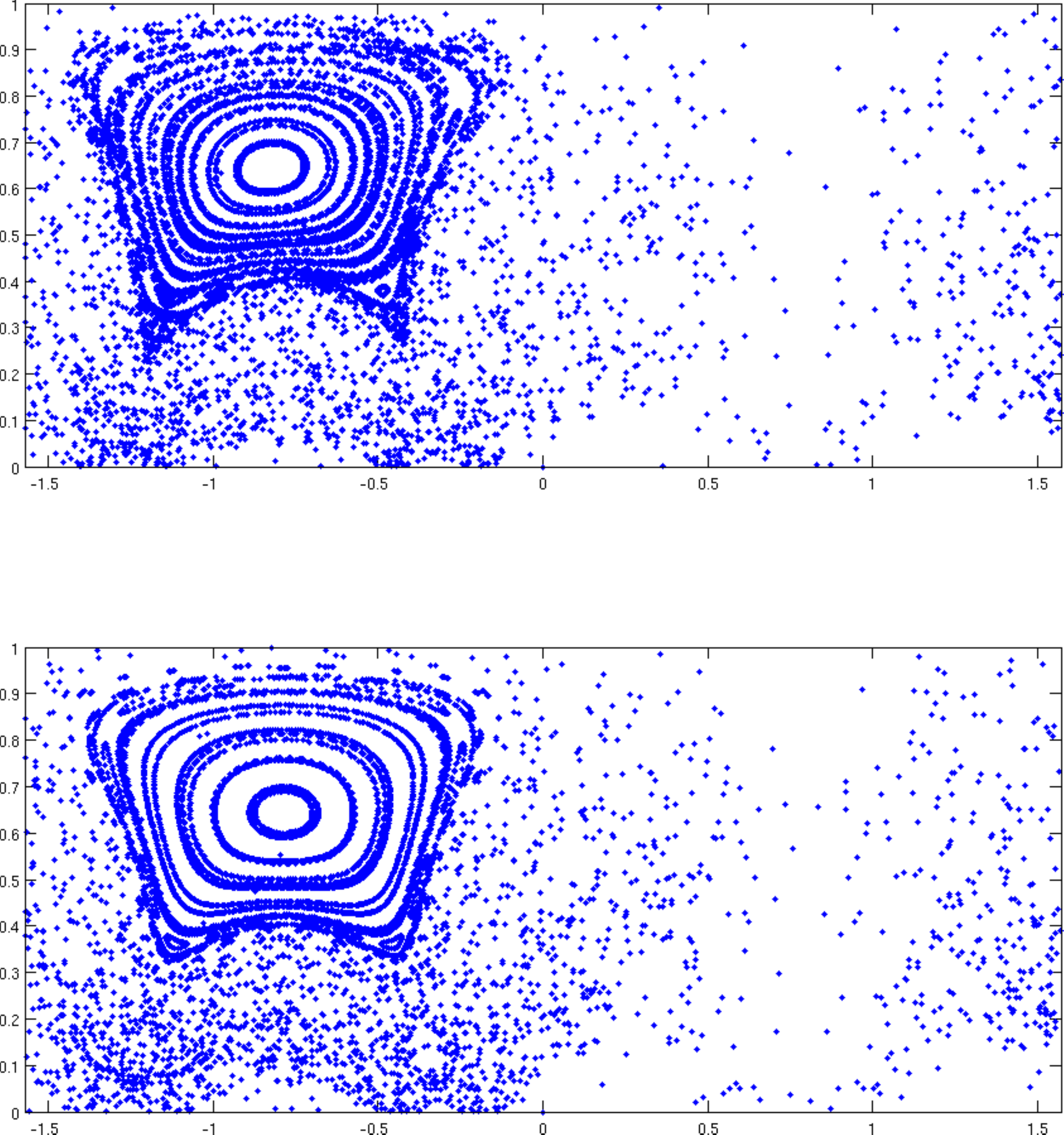}
\caption{Poincar\'e sections plotted for $m = .9$ (top) and $m = 1$ (bottom).}
\label{manypoin5}
\end{center}
\end{figure}
}

\afterpage{\clearpage}

\section{The Rhomboidal Four-Degree-of-Freedom Symmetric-Mass Orbit}\label{sec4df}

\subsection{Description and Existence} \label{sec4dfper}

We now consider the planar Newtonian $4$-body problem with bodies located at 
\begin{equation}
(x_1, x_2), (x_3, x_4), (-x_1, -x_2), (-x_3, -x_4)
\label{config}
\end{equation}
and masses $1$, $m$, $1$, $m$ respectively for some $m \in (0,1]$.  (It is important to note here that $x_2$ does not correspond to $x_2$ from the previous section.)  For the periodic orbit, the bodies still travel along the $x$ and $y$ axes, forming the vertices of a rhombus at all times away from collision, with the same behaviors of the 2DF orbit (such as zero momentum of non-colliding bodies at collision time) still holding. \\

The Hamiltonian for this system is given by $H = K - U$, where
$$K = \frac{1}{4}\left(w_1^2 + w_2^2 \right) + \frac{1}{4m}\left(w_3^2 + w_4^2\right)$$
where the $w_i$ are the conjugate momenta defined by
$$w_1 = 2\dot{x}_1, w_2 = 2\dot{x}_2, w_3 = 2m\dot{x}_3, w_4 = 2m\dot{x}_4,$$
and
\begin{align*}
U &= \frac{1}{2\sqrt{x_1^2 + x_2^2}} + \frac{m^2}{2\sqrt{x_3^2 + x_4^2}} \\
&+ \frac{2m}{\sqrt{(x_3-x_1)^2 + (x_4-x_2)^2}} + \frac{2m}{\sqrt{(x_3+x_1)^2 + (x_4+x_2)^2}}.
\end{align*}

\noindent The angular momentum for the system is given by
$$A = x_1w_2 - x_2w_1 + x_3w_4 - x_4w_3.$$

We can regularize the system under a change of spatial variables and a re-scaling of time.  Define
$$F = w_1(Q_1^2 - Q_2^2) + 2w_2Q_1Q_2 + 2w_3Q_3Q_4 + w_4(Q_4^2 - Q_3^2).$$
Then $F$ induces the canonical change of variables $(x_i, w_i) \leftrightarrow (Q_i, P_i)$ given by
\begin{align*}
x_1 &= Q_1^2 - Q_2^2 \ \ \ &P_1 &= 2w_1Q_1 + 2w_2Q_2 \\
x_2 &= 2Q_1Q_2 \ \ \ &P_2 &= -2w_1Q_2 + 2w_2Q_1 \\
x_3 &= 2Q_3Q_4 \ \ \ &P_3 &= 2w_3Q_4 - 2w_4Q_3 \\
x_4 &= Q_4^2 - Q_3^2 \ \ \ &P_4 &= 2w_3Q_3 + 2w_4Q_4.
\end{align*}

Each of the $P_i$ is linear in $w_i$.  Solving the resulting system of equations yields
\begin{equation*}
\begin{bmatrix}
w_1 \\
w_2
\end{bmatrix}
= \frac{1}{2(Q_1^2 + Q_2^2)}
\begin{bmatrix}
Q_1 & -Q_2 \\
Q_2 & Q_1
\end{bmatrix}
\begin{bmatrix}
P_1 \\
P_2
\end{bmatrix}
\end{equation*}
and
\begin{equation*}
\begin{bmatrix}

w_3 \\
w_4
\end{bmatrix}
= \frac{1}{2(Q_3^2 + Q_4^2)}
\begin{bmatrix}
Q_4 & Q_3 \\
-Q_3 & Q_4
\end{bmatrix}
\begin{bmatrix}
P_3 \\
P_4
\end{bmatrix}.
\end{equation*}


\noindent Setting $\mathbf{Q} = Q_1^2Q_3Q_4 - Q_2^2Q_3Q_4 - Q_1Q_2Q_3^2 + Q_1Q_2Q_4^2$, we now have
$$K = \frac{1}{16} \left(\frac{P_1^2 + P_2^2}{Q_1^2 + Q_2^2} \right) + \frac{1}{16m} \left(\frac{P_3^2 + P_4^2}{Q_3^2 + Q_4^2} \right),$$
\begin{align*}
U &= \frac{1}{2(Q_1^2 + Q_2^2)} + \frac{m^2}{2(Q_3^2 + Q_4^2)} \\
&+ \frac{2m}{\sqrt{(Q_1^2 + Q_2^2)^2 + (Q_3^2 + Q_4^2)^2 - 4\mathbf{Q}}} + \frac{2m}{\sqrt{(Q_1^2 + Q_2^2)^2 + (Q_3^2 + Q_4^2)^2 + 4\mathbf{Q}}},
\end{align*}
and
$$A = \frac{1}{2}\left(Q_1P_2 - Q_2P_1 + Q_3P_4 - Q_4P_3\right).$$

\noindent We can regularize the collisions at the origin by multiplying by a change of time satisfying $\frac{dt}{ds} = (Q_1^2 + Q_2^2)(Q_3^2 + Q_4^2)$.  At the time of collision between the two bodies of mass 1, we have $Q_1 = Q_2 = 0$.  The condition
$$\Gamma = \frac{dt}{ds}(H - E) = 0$$
then yields
$$(Q_3^2 + Q_4^2)\left(\frac{P_1^2 + P_2^2}{16} - \frac{1}{2} \right) = 0$$
and so at collision the momenta $P_1$ and $P_2$ are both finite and satisfy $P_1^2 + P_2^2 = 8$.  Similarly, when $Q_3 = Q_4 = 0$, we get
$$(Q_1^2 + Q_2^2)\left(\frac{P_3^2 + P_4^2}{16m} - \frac{m^2}{2}\right) = 0$$
so the momenta $P_3$ and $P_4$ are both finite and satisfy $P_3^2 + P_4^2 = 8m^3$. \\

Let $\mathcal{A}$ denote the set where
$$x_2 = x_3 = w_2 = w_3 = 0\text{, } x_1 > 0, \text{ and } x_4 > 0.$$
This corresponds to the regularized coordinates
\begin{equation}\label{aRegConds}
Q_2 = Q_3 = P_2 = P_3 = 0.
\end{equation}
Then, when $\mathcal{A}$ holds, the four bodies and their respective momenta lie on the $x$- and $y$-axes, as in the two-degree of freedom (2DF) problem.  We also have
$$Q_i^2 = x_i, \quad w_i = \frac{P_i}{2Q_i} \text{ for } i = 1, 4,$$
which are the same coordinate transformations used in \cite{bibYan1} and in our work in Section \ref{sec2df}.  Furthermore,  we have
$$\dot{Q}_2 \big|_{\mathcal{A}} = \dot{Q}_3 \big|_{\mathcal{A}} = \dot{P}_2 \big|_{\mathcal{A}} = \dot{P}_3 \big|_{\mathcal{A}} =0,$$
\noindent so $\mathcal{A}$ is invariant, and corresponds to the 2DF rhomboidal configuration.  Hence, the 2DF problem embeds nicely into the 4DF problem, and initial conditions from the 2DF problem can also be used to study the 4DF orbit.  This result, combined with the existence of the 2DF orbit from \cite{bibShib1}, and \cite{bibMartinez}, demonstrates the analytic existence of the 4DF rhomboidal orbit.

\subsection{Symmetries of the Rhomboidal Four-Degree-of-Freedom Orbit}\label{syms}

Let
\begin{equation*}
G = 
\begin{bmatrix}
1 & 0 \\ 0 & -1
\end{bmatrix},
\end{equation*}
and define the block matrix
\begin{equation}\label{sMatrix4df}
S = 
\begin{bmatrix}
-G & 0 & 0 & 0 \\
0 & -G & 0 & 0 \\
0 & 0 & G & 0 \\
0 & 0 & 0 & G
\end{bmatrix},
\end{equation}
where $0$ represents the $2 \times 2$ identity matrix.  Then we have
$$S^2 = (-S)^2 = I$$
Hence, $S$ and $-S$ generate a group isomorphic to the Klein four group.  For fixed values of $m$ and $E$, we have
$$\Gamma \circ (\pm S) = \Gamma$$
so $\pm S$ generate a Klein-four symmetry group for $\Gamma$ as well. \\

\begin{theorem}
Let $\gamma$ be a solution to the Hamiltonian system defined by $\Gamma$ for some fixed values of $m \in (0,1]$ and $E < 0$ such that 
$$\gamma(0) = (0,0,0,\zeta_4,\sqrt{8},0,0,0)$$
and
$$\gamma(s_0) = (\zeta_1,0,0,0,0,0,0,\sqrt{8m^3}).$$
(In other words, $\gamma(0)$ corresponds to collision between the two bodies of mass 1, and $\gamma(s_0)$ corresponds to collision between the two bodies of mass $m$.)  Then $\gamma$ extends to a $T = 4s_0$-periodic solution of the same Hamiltonian system, wherein $S$ and $-S$ are time-reversing symmetries for the orbit.  Specifically, for all time $s$, we have
$$-S\gamma(T/2 - s) = \gamma(s) \text{ and }= S\gamma(T - s) = \gamma(s).$$
\end{theorem}

\subsection{Stability Calculations with Symmetry}\label{stabcalc}

Continuing from the brief introduction given in \ref{LinStab1}, if $Y$ is the fundamental matrix solution of
\begin{equation}\label{Linearized2}
\xi' = JD^2\Gamma(\gamma(s))\xi, \quad \xi(0) = Y_0
\end{equation}
for some invertible matrix $Y_0$, then by definition of $X(s)$, $Y(s) = X(s)Y_0$, implying $X(T) = Y(T)Y_0^{-1}$.  Then we have
$$X(T) = Y(T)Y_0^{-1} = Y_0(Y_0^{-1}Y(T))Y_0^{-1}$$
and so $X(T)$ and $Y_0^{-1}Y(T)$ are similar, and stability can be determined by the eigenvalues of either.  For our purposes, the latter will be preferred.\\

The following can be found in \cite{bibRoberts1}:

\begin{lemma}
Suppose $\gamma(s)$ is a $T$-periodic solution of a Hamiltonian system with Hamiltonian $\Gamma$ and a time-reversing symmetry $S$ such that:
\begin{enumerate}
\item[(i)] For some $n \in \mathbb{N}$, $\gamma(-s + T/N) = S(\gamma(s))$ for all $s$;
\item[(ii)] $\Gamma(Sz) = \Gamma(z)$;
\item[(iii)] $SJ = -JS$;
\item[(iv)] $S$ is orthogonal.
\end{enumerate}
Then the fundamental matrix solution $X(s)$ satisfies $$X(-s + T/N) = SX(s)S^T(X(T/N)).$$
\end{lemma}

\noindent Note that the matrix $S$ given in (\ref{sMatrix4df}) satisfies all the required hypotheses.

\begin{corollary}
Under the same hypotheses,
$$X(T/N) = SB^{-1}S^TB \text{ where } B = X(T/2N).$$
\end{corollary}

\begin{corollary}\label{corFactor}
If $Y(s)$ is the fundamental matrix solution to (\ref{Linearized2}), then
$$Y(-s + T/N) = SY(s)Y_0^{-1}S^TY(T/N)$$
and
$$Y(T/N) = SY_0B^{-1}S^TB \text{ where } B = Y(T/2N).$$
\end{corollary}

\noindent Similar results for time-preserving symmetries are also presented in \cite{bibRoberts1}, but are not needed for this orbit.  Using these results may allow the computation of the eigenvalues (hence stability) to be accomplished using only a fraction of the orbit.  Applying Corollary \ref{corFactor} with $N = 2$, $S$ as defined in (\ref{sMatrix4df}), and noting that $S^T = S$ yields
$$Y(T/2) = SY_0Y(T/4)^{-1}SY(T/4).$$
Similarly, if $N = 1$, since $S^2 = I$, we get
\begin{align*}
Y(T) &= -SY_0Y(T/2)^{-1}(-S)Y(T/2) \\
     &= SY_0[SY_0Y(T/4)^{-1}SY(T/4)]^{-1}S[SY_0Y(T/4)^{-1}SY(T/4)] \\   
     &= SY_0Y(T/4)^{-1}SY(T/4)Y_0^{-1}SY_0Y(T/4)^{-1}SY(T/4).
\end{align*}
This yields
\begin{align*}
Y_0^{-1}Y(T) &= Y_0^{-1}SY_0Y(T/4)^{-1}SY(T/4)Y_0^{-1}SY_0Y(T/4)^{-1}SY(T/4) \\
&= [Y_0^{-1}SY_0Y(T/4)^{-1}SY(T/4)]^2 \\
&= W^2
\end{align*}
with $W =Y_0^{-1}SY_0Y(T/4)^{-1}SY(T/4)$.  Hence, in order to analyze the stability of the orbit, we need only compute the eigenvalues of $Y$ along a quarter of the orbit. \\

Again, from \cite{bibRoberts1}:

\begin{lemma}
For a symplectic matrix $W$, suppose there is a matrix $K$ such that
\begin{equation*}
\frac{1}{2}(W + W^{-1}) =
\begin{bmatrix}
K^T & 0 \\
0 & K
\end{bmatrix}.
\end{equation*}
Then $W$ is stable if and only if all of the eigenvalues of $K$ are real and have absolute value less than or equal to 1.
\end{lemma}

We now show that there is an appropriate choice of $Y_0$ for which $W$ has the required form, further reducing the stability calculations for the orbit.  If we let
\begin{equation*}
\Lambda =
\begin{bmatrix}
I & 0 \\
0 & -I 
\end{bmatrix},
\end{equation*}
then setting
\begin{equation}\label{y0form}
Y_0 =
\left[
\begin{array}{cccc|cccc}
0 & 0 & 0 & 0 & 1 & 0 & 0 & 0 \\
0 & 0 & 1 & 0 & 0 & 0 & 0 & 0 \\
0 & 0 & 0 & 0 & 0 & 1 & 0 & 0 \\
0 & 0 & 0 & 1 & 0 & 0 & 0 & 0 \\
\hline
-1 & 0 & 0 & 0 & 0 & 0 & 0 & 0 \\
0 & 0 & 0 & 0 & 0 & 0 & 1 & 0 \\
0 & -1 & 0 & 0 & 0 & 0 & 0 & 0 \\
0 & 0 & 0 & 0 & 0 & 0 & 0 & 1 

\end{array}
\right]
\end{equation}
yields $-Y_0^{-1}SY_0 = \Lambda$.  (The lines here are provided for ease in reading.  Many of our later analysis will involve breaking $8 \times 8$ matrices down into $4 \times 4$ blocks.)  Furthermore, it is easy to check that $Y_0$ is both orthogonal and symplectic.  If we set $D = -B^{-1}SB$ for $B = Y(T/4)$, we then have
$$W = \Lambda D.$$
Also, since $\Lambda^2 = D^2 = I$, we know immediately that
$$W^{-1} = D\Lambda.$$
Since $B = Y(T/4)$ is symplectic, setting
\begin{equation*}
B =
\begin{bmatrix}
B_1 & B_2 \\
B_3 & B_4
\end{bmatrix}
\text{ and }
S =
\begin{bmatrix}
S_1 & 0 \\
0 & -S_1
\end{bmatrix}
\end{equation*}
gives
\begin{align*}
D &= -B^{-1}SB \\
&= -
\begin{bmatrix}
B_4^T & -B_2^T \\
-B_3^T & B_1^T
\end{bmatrix}
\begin{bmatrix}
S_1 & 0 \\
0 & -S_1
\end{bmatrix}
\begin{bmatrix}
B_1 & B_2 \\
B_3 & B_4
\end{bmatrix}\\
&=-
\begin{bmatrix}
B_4^TS_1B_1 + B_2^TS_1B_3 & B_4^TS_1B_2 + B_2^TS_1B_4 \\
-B_3^TS_1B_1 - B_1^TS_1B_3 & -B_3^TS_1B_2 - B_1^TS_1B_4 \\
\end{bmatrix}\\
&= -
\begin{bmatrix}
K^T & L_1 \\
-L_2 & K
\end{bmatrix}.
\end{align*}
Thus,
\begin{equation*}
W = \Lambda D = 
\begin{bmatrix}
K^T & L_1 \\
L_2 & K
\end{bmatrix}.
\end{equation*}
Similarly, we find that
\begin{equation*}
W^{-1} = D \Lambda = 
\begin{bmatrix}
K^T & -L_1 \\
-L_2 & K
\end{bmatrix}.
\end{equation*}
Thus, we have
\begin{equation}
\frac{1}{2}\left(W + W^{-1}\right) = 
\begin{bmatrix}
K^T & 0 \\
0 & K
\end{bmatrix}
\label{WKrelate}
\end{equation}
for some $4 \times 4$ matrix $K$. \\

\textit{Remark:} The given matrix $Y_0$ in (\ref{y0form}) is not unique.  Different choices of $Y_0$ will give different properties of the monodromy matrix.  It is also worth noting that $Y_0$ is independent of the value of $m$ for this orbit, which is not always true (see \cite{bibBMS}.)\\

We can give formulas for the entries of $K$ in terms of $W$.  Since $B$ is symplectic, we have $J = B^TJB$, and hence
$$B^{-1} = -JB^T J.$$
Using $W = \Lambda D$ for $D = -B^{-1}SB$ and the relation $-SJ = JS$, we find
\begin{align*}
W &= \Lambda (-B^{-1}SB) \\
&= \Lambda JB^T JSB \\
&= -\Lambda JB^T SJB.
\end{align*}
Directly computing $\Lambda J$ and using the block form of $B$, we find that
\begin{equation*}
(\Lambda J) B^T = -
\begin{bmatrix}
0 & I \\
I & 0
\end{bmatrix}
\begin{bmatrix}
B_1^T & B_3^T \\
B_2^T & B_4^T
\end{bmatrix}
= -
\begin{bmatrix}
B_2^T & B_4^T \\
B_1^T & B_3^T
\end{bmatrix}
.
\end{equation*}
Define $\text{col}_i(-SJB)$ to be the $i$th column of the matrix $-SJB$.  Then we have $\text{col}_i(-SJB) = -SGc_i$ where $c_i$ is the $i$th column of $B$.  Using the above two formulas, this implies that the $(i,j)$ entry of $W$ is given by $-c_i^T S J C_j$.  Equation (\ref{WKrelate}) shows that the $(i,j)$ entry of $K$ is the $(i+4, j+4)$ entry of $W$.  Hence, 
\begin{equation}
K = 
\begin{bmatrix}
-c_1^T SJc_5 & -c_1^T SJc_6 & -c_1^T SJc_7 & -c_1^T SJc_8\\
-c_2^T SJc_5 & -c_2^T SJc_6 & -c_2^T SJc_7 & -c_2^T SJc_8\\
-c_3^T SJc_5 & -c_3^T SJc_6 & -c_3^T SJc_7 & -c_3^T SJc_8\\
-c_4^T SJc_5 & -c_4^T SJc_6 & -c_4^T SJc_7 & -c_4^T SJc_8\\
\end{bmatrix}.
\label{Kentries}
\end{equation}

\textit{Remark:} Computing the entries of $K$ this way will allow us to bypass computing $W^{-1}$.  This is preferred as a numerical method as $W$ may be very poorly conditioned. \\

\subsection{\textit{A priori} Determination of Values of $K$}\label{apriori}

With a bit more work, we can show some additional properties of the matrix $K$.  Let $v = Y_0^{-1}\gamma^{\ '}(0)/||\gamma^{\ '}(0)||$ or, equivalently, $Y_0^T\gamma^{\ '}(0)/||\gamma^{\ '}(0)||$.  By Corollary \ref{corFactor}, since $Y_0$ is orthogonal and $S$ is symmetric, we have
$$W = Y_0^{-1}SY_0B^{-1}SB = Y_0^{-1}SY_0B^{-1}S^TB = Y_0^TY(T/2).$$
Since $\gamma^{\ '}(s)$ is a solution of the linearized equations $\dot{\xi} = JD^2\Gamma(\gamma(s))\xi$ and $\gamma^{\ '}(0) = Y(0)Y_0^{-1}\gamma^{\ '}(0) = Y(0)v$, we also know that $\gamma^{\ '}(s) = Y(s)Y_0^{-1}\gamma^{\ '}(0) = Y(s)v$.  This implies
\begin{equation}\label{eigeq1}
Y_0^{-1}\gamma^{\ '}(T/2) = Y_0^TY(T/2)v = Wv.
\end{equation}
By the symmetry $\gamma(s) = -S\gamma(T/2 - s)$, we also have $\gamma^{\ '}(s) = S\gamma^{\ '}(T/2 - s)$.  Setting $s = 0$ in this setting tells us that $\gamma^{\ '}(0) = S\gamma^{\ '}(T/2)$.  Since
$$\gamma^{\ '}(0) = (\alpha, 0, 0, 0, 0, 0, 0, 0)$$
for some real number $\alpha$, we have $-S\gamma^{\ '}(0) = \gamma^{\ '}(0)$.  Thus
\begin{equation}\label{eigeq2}
Y_0^{-1}\gamma^{\ '}(T/2) = Y_0^TS\gamma^{\ '}(0) = -Y_0^T\gamma^{\ '}(0) = -v.
\end{equation}
Combining (\ref{eigeq1}) and (\ref{eigeq2}) gives $Wv = -v$, and so $-1$ is an eigenvalue of $W$ with eigenvector $v$.  By definition, we have that
$$v = Y_0^T\gamma^{\ '}(0)/||\gamma^{\ '}(0)|| = (0, 0, 0, 0, 1, 0, 0, 0).$$
From the form of $W$, this implies that
\begin{equation*}
K
\begin{bmatrix}
1 \\
0 \\
0 \\
0
\end{bmatrix}
=
\begin{bmatrix}
-1 \\
0 \\
0 \\
0 \\
\end{bmatrix}
\end{equation*}
so the first column of $K$ must be $[-1,0,0,0]^T$. \\

\textit{Remark:} The choice of $Y_0$ in (\ref{y0form}) forces $v$ to be in the eigenspace of $K$ corresponding to the eigenvalue $-1$.  Alternative choices of $Y_0$ can result in $Kv = v$.  \\

In numerically computing $K$, additional patterns arose in the entries.  These patterns can be explained and verified analytically.  Let $\mathcal{M}$ denote the set of matrices of the form
\begin{equation*}
\begin{bmatrix}
m_{11} & 0 & 0 & m_{14}\\
0 & m_{22} & m_{23} & 0\\
0 & m_{32} & m_{33} & 0\\
m_{41} & 0 & 0 & m_{44}\\
\end{bmatrix}
\end{equation*}
where all of the listed $m_{ij} \in \mathbb{R}$.  (We allow for $m_{ij} = 0$.)  Then $\mathcal{M}$ is closed under multiplication.  Let $\mathcal{M}_2$ denote the set of $8 \times 8$ matrices whose $4 \times 4$ blocks are in $\mathcal{M}$.  That is to say, $\mathcal{M}_2$ consists of matrices of the form
\begin{equation*}
\begin{bmatrix}
M_1 & M_2\\
M_3 & M_4\\
\end{bmatrix}
\end{equation*}
where each of the $M_i \in \mathcal{M}$. Then $\mathcal{M}_2$ is closed under multiplication as well.  Furthermore, it is readily verified that each of $J$, $S$, and $Y_0$ are in $\mathcal{M}_2$.  Using a computer algebra system, we find that the matrix $D^2\Gamma$ is of the form
\begin{equation*}
\left[
\begin{array}{cccc|cccc}
* & a & a & * & 0 & 0 & a & * \\
a & * & * & a & 0 & 0 & a & a \\
a & * & * & a & a & a & 0 & 0 \\
* & a & a & * & * & a & 0 & 0 \\
\hline
0 & 0 & a & * & * & 0 & 0 & 0 \\
0 & 0 & a & a & 0 & * & 0 & 0 \\
a & a & 0 & 0 & 0 & 0 & * & 0 \\
* & a & 0 & 0 & 0 & 0 & 0 & *
\end{array}
\right].
\end{equation*}
Here, the zeros denote entries for which the mixed partials evaluate to zero identically, and the entries denoted $a$ are entries for which the mixed partials evaluate to zero assuming the conditions given by (\ref{aRegConds}) which hold along the periodic orbit $\gamma(s)$.  Under such conditions, we have $D^2\Gamma \in \mathcal{M}_2$.  \\

\begin{lemma}
If $M \in \mathcal{M}_2$, then the system of differential equations given by
$$\eta' = M\eta$$
and initial condition
$$\eta(0) = (*, 0, 0, *, *, 0, 0, *)^T$$
has solutions of the form
$$\eta(s) = (f_1(s), 0, 0, f_4(s), f_5(s), 0, 0, f_8(s))^T$$
\end{lemma}
\begin{proof}
We verify that $M\eta$ has the proper form.  Note that 
\begin{equation*}
\left[
\begin{array}{cccc|cccc}
* & 0 & 0 & * & * & 0 & 0 & * \\
0 & * & * & 0 & 0 & * & * & 0 \\
0 & * & * & 0 & 0 & * & * & 0 \\
* & 0 & 0 & * & * & 0 & 0 & * \\
\hline
* & 0 & 0 & * & * & 0 & 0 & * \\
0 & * & * & 0 & 0 & * & * & 0 \\
0 & * & * & 0 & 0 & * & * & 0 \\
* & 0 & 0 & * & * & 0 & 0 & *
\end{array}
\right]
\begin{bmatrix}
* \\
0 \\
0 \\
* \\
* \\
0 \\
0 \\
* \\
\end{bmatrix}
=
\begin{bmatrix}
* \\
0 \\
0 \\
* \\
* \\
0 \\
0 \\
* \\
\end{bmatrix}.
\end{equation*}
Hence, the zeros in the 2nd, 3rd, 6th, and 7th are preserved under multiplication by $M$.  So 
$$\eta(s) = (f_1(s), 0, 0, f_4(s), f_5(s), 0, 0, f_8(s))^T$$
is a solution of $\eta' = M\eta$.  Existence and uniqueness of solutions implies that $\eta(s)$ is the only solution of the system.
\end{proof}

\begin{corollary}
If $M \in \mathcal{M}_2$, then the system of differential equations given by
$$\eta' = M\eta$$
and initial condition
$$\eta(0) = (0, *, *, 0, 0, *, *, 0)^T$$
has solutions of the form
$$\eta(s) = (0, f_2(s), f_3(s), 0, 0, f_6(s), f_7(s), 0)^T.$$
\end{corollary}

\begin{corollary}
If $\xi(0) \in \mathcal{M}_2$, then the solution to the system of linearized equations given by (\ref{Linearized2}) satisfies $\xi(s) \in \mathcal{M}_2$ for all $s$.
\end{corollary}

\textit{Remarks:}

\begin{enumerate}
\item[(i)] In terms of the 4DF Rhomboidal orbit, this form very nicely decomposes phase space into a direct sum of the subspaces $\mathcal{A} = \{Q_2 = Q_3 = P_2 = P_3 = 0\}$ and $\mathcal{A}^\perp = \{Q_1 = Q_4 = P_1 = P_4 = 0\}$.  This decomposition is due in part to the coordinate transformation we chose.  The choice of notation for $\mathcal{A}^\perp$ is appropriate in that $\mathcal{A}^\perp$ and $\mathcal{A}$ are orthogonal complements in $\mathbb{R}^8$.  The two subspaces are also skew-orthogonal: if $a_1 \in \mathcal{A}$ and $a_2 \in \mathcal{A}^\perp$, then $a_1^TJa_2 = 0$.

\item[(ii)] Matrices of the form $\mathcal{M}$ and $\mathcal{M}_2$ are similar to the diamond product discussed in \cite{bibLong}.  Specifically, $\Sigma^{-1} M \Sigma = A_1 \Diamond A_2$ for some matrices $A_1$ and $A_2$, where $M \in \mathcal{M}_2$ and $\Sigma$ is the permutation matrix corresponding to $\sigma = (1 \ 2 \ 3 \ 4 \ 5 \ 6 \ 7 \ 8).$ Furthermore, one of these two matrices corresponds to the 2DF setting.

\item[(iii)] The particular choice of $Y_0$ given in (\ref{y0form}) is again important for this argument.
\end{enumerate}

Assuming the initial condition $\xi(0) = Y_0$, then $\xi(s) \in \mathcal{M}_2$ for all $s$.  Hence, $W \in \mathcal{M}_2$.  It is easily shown that if $M \in \mathcal{M}_2$ is invertible, then $M^{-1} \in \mathcal{M}_2$.  By Equation (\ref{WKrelate}), it follows that $K \in \mathcal{M}$, and so we know that
\begin{equation}\label{Kform}
K = 
\begin{bmatrix}
-1 & 0 & 0 & * \\
0 & a & b & 0 \\
0 & c & d & 0 \\
0 & 0 & 0 & e
\end{bmatrix}
\end{equation}
by the above result about one of the eigenvectors of $W$.  Hence, the remaining three eigenvalues of $K$ are simply those of the central $2 \times 2$ matrix together with $e$.  \\

\textit{Remark:} Owing to the decomposition mentioned earlier, the position of $e$ in the matrix $K$ indicates that it should be an eigenvalue corresponding to the behavior of the orbit in the subspace $\mathcal{A}$, along with the trivial eigenvalue $-1$ from the $(1,1)$ position.  Hence, computing the linear stability of the 4DF orbit automatically gives the stability of the 2DF orbit.\\

The coordinate changes that we used to regularize the system did not ``factor out'' the angular momentum (as the polar symplectic transformation does), and so we expect that $K$ will have another trivial eigenvalue $\pm 1$ corresponding to this integral.  To demonstrate this, let $\gamma(s)$ again represent the periodic orbit with period $T$ as above, and define $\hat{v}(s) = \nabla A(\gamma(s))$.  As in \cite{bibMHO}, p. 134, Lemma 7, we consider $\hat{v}(s)$ as a left eigenvector of a particular matrix related to the monodromy matrix.  We find that
$$\hat{v} = \frac{1}{2}\left(P_2, -P_1, P_4, -P_3, -Q_2, Q_1, -Q_4, Q_3 \right).$$
Since
$$\gamma(0) = (0,0,0,\zeta_4,\sqrt{8},0,0,0),$$
we know that
$$\hat{v}(0) = \frac{1}{2} \left(0,-\sqrt{8}, 0, 0, 0, 0, -\zeta_4, 0\right).$$
If $\phi(s,z)$ represents the solution to the system of linearized differential equations with initial condition $z$, we know that
$$A(\phi(s,z)) = A(z).$$
Differentiating with respect to $z$ gives
$$\nabla A(\phi(s,z))\frac{\partial \phi}{\partial z}(s,z) = \nabla A(z)$$
or, equivalently
$$\hat{v}(s)X(s) = \hat{v}(0)$$
where $X(s)$ is the fundamental matrix solution.  Setting $s = T/2$ and substituting $X(T/2) = Y_0(Y_0^{-1}Y(T/2))Y_0^{-1}$ gives
$$\hat{v}(T/2)Y_0(Y_0^{-1}Y(T/2))Y_0^{-1} = \hat{v}(0)$$
and so
$$\hat{v}(T/2)Y_0(Y_0^{-1}Y(T/2)) = (\hat{v}(T/2)Y_0)W= \hat{v}(0)Y_0.$$
By the symmetry of the orbit, we know that $\gamma(T/2) = -\gamma(0)$, which gives
$$\hat{v}(T/2) = -\hat{v}(0),$$
and therefore
$$(\hat{v}(0)Y_0)W= -\hat{v}(0)Y_0.$$
Hence $\hat{v}(0)Y_0$ is a left eigenvector for $W$ with eigenvalue $-1$.  \\

This additional eigenvector and eigenvalue gives us more information about the structure of $W$, and hence, of $K$.  We know that
\begin{equation*}
\hat{v}Y_0
\begin{bmatrix}
K^T & 0 \\ 0 & K
\end{bmatrix}
= - \hat{v}Y_0,
\end{equation*}
We readily compute $\hat{v}(0)Y_0 = \frac{1}{2}(0,-\zeta_4,\sqrt{8},0,0,0,0,0)$.  From this, we know that
$$(0,-\zeta_4,\sqrt{8},0)K^T = - (0,-\zeta_4,\sqrt{8},0).$$
Since $K \in \mathcal{M}$, this requires that the additional $-1$ eigenvalue comes from the central $2 \times 2$ block in $K$.  Furthermore, this imposes some relations on the entries $a, b, c, d$ in (\ref{Kform}).  In particular,
$$b = \frac{(a + 1)\zeta_4}{\sqrt{8}},$$
$$c = \frac{(d + 1)\sqrt{8}}{\zeta_4}.$$
\textit{Remarks:} 
\begin{enumerate}
\item[(i)] Since $K$ is real-valued, this result, along with other results about the form of $K$, force all of the eigenvalues of $K$ to be real numbers.
\item[(ii)] This analysis is an improvement over work done in \cite{bibBMS}, in which the $-1$ eigenvalue corresponding to angular momentum showed up numerically but could not be factored out \textit{a priori}.  This is most likely due to the relative simplicity of the rhomboidal orbit.
\end{enumerate}

\subsection{Results}\label{Results}

We numerically obtain the matrix $W$ by a numerical integration of the linearized systems (\ref{Linearized2}) and the initial conditions computed in Section \ref{Numerics}.  The values of $a$, $d$, and $e$ in the matrix $K$, as given in (\ref{Kform}), are readily computed using (\ref{Kentries}) and the computed value of $\zeta_4$ from Section \ref{Numerics}.  The resulting eigenvalue calculations are represented in Figures \ref{unstabplot} and \ref{stabplot}. \\

\begin{figure}[h]
\begin{center}
\includegraphics[scale=.45]{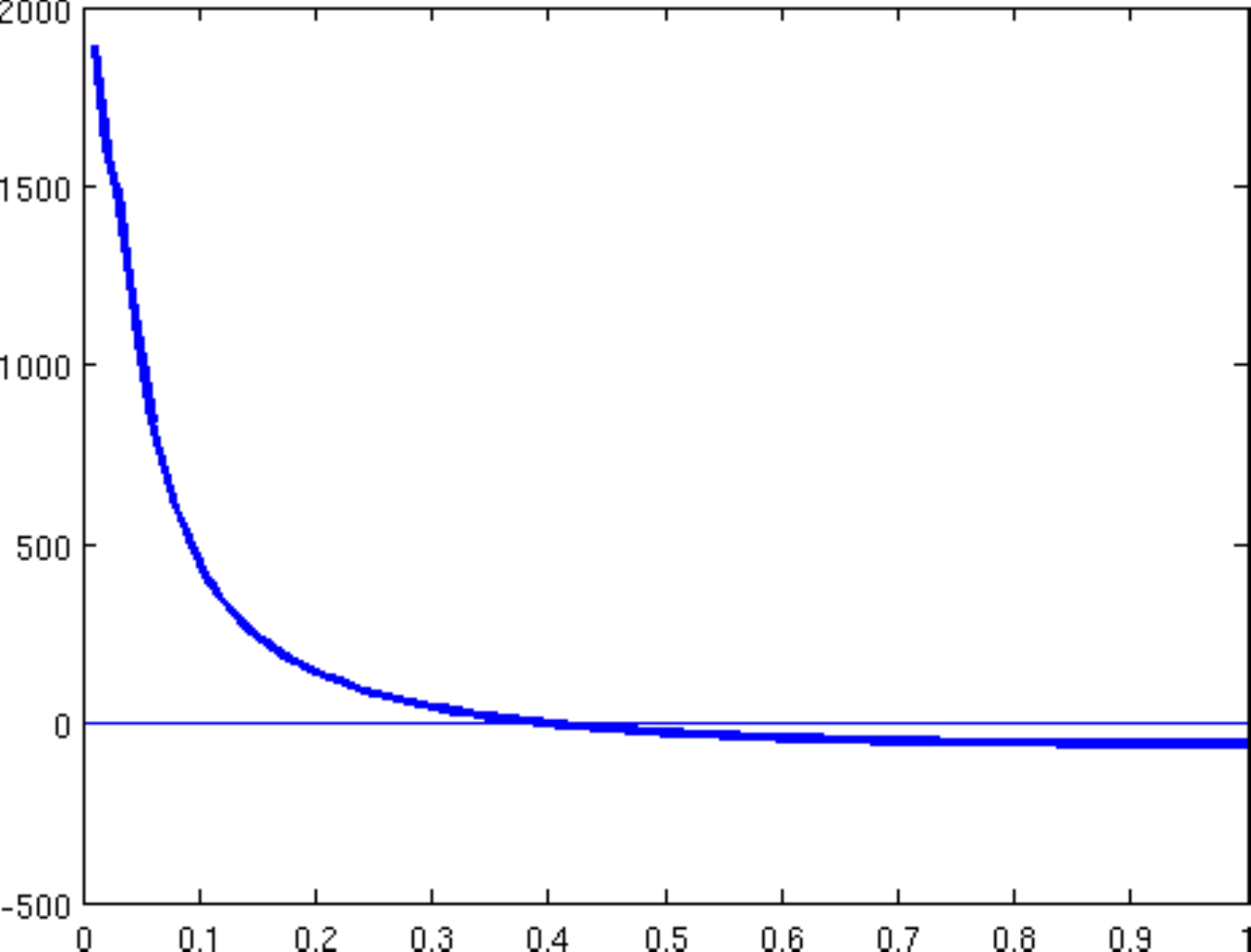}
\caption{A plot of the nontrivial eigenvalue of the central $2 \times 2$ submatrix of $K$ as a function of $m$.  This eigenvalue crosses the $y$-axis for some value of $m \approx 0.4$.
}
\label{unstabplot} 
\end{center}
\end{figure}

\begin{figure}[h]
\begin{center}
\includegraphics[scale=.45]{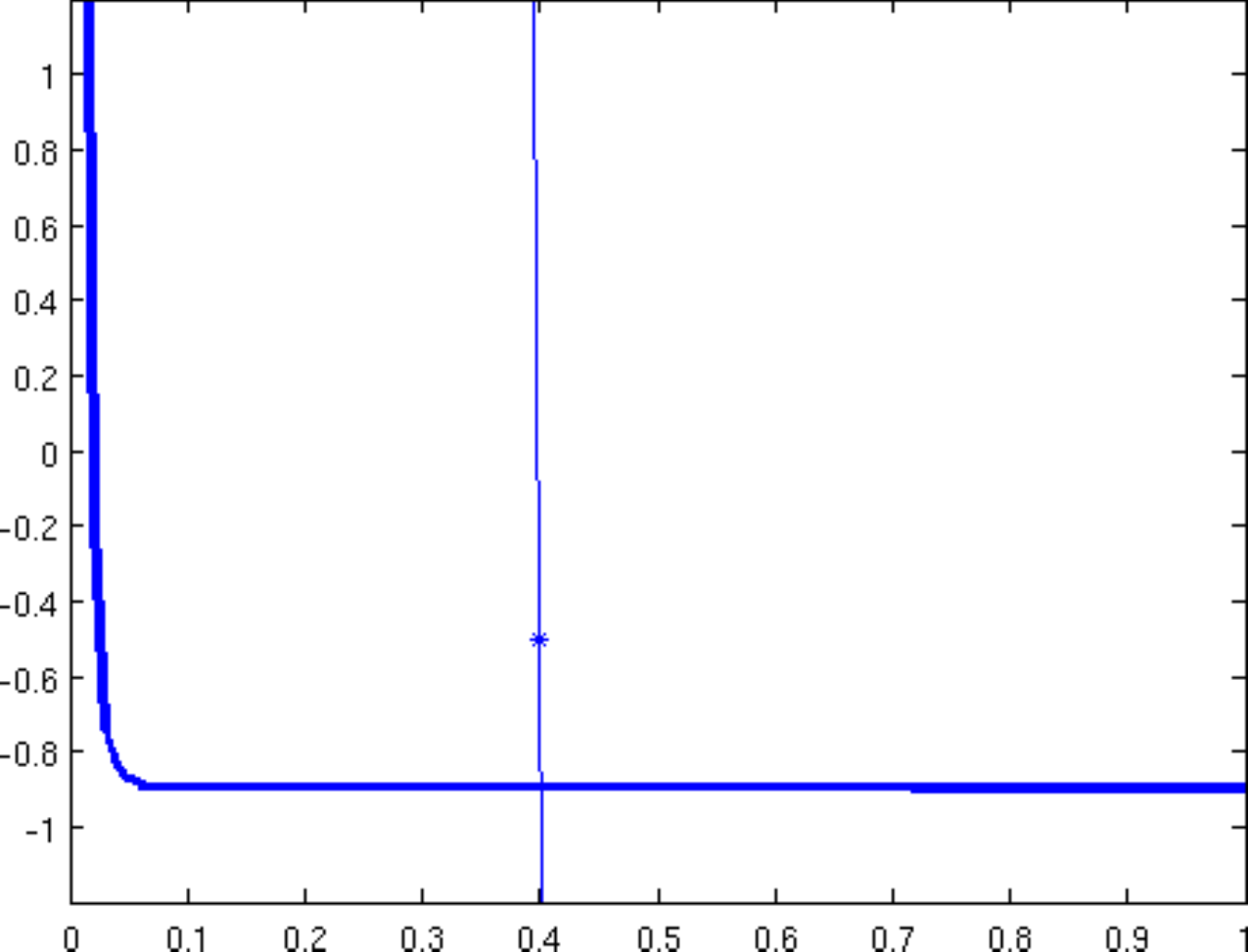}d
\caption{A plot of the nontrivial eigenvalues of $K$ as functions of $m$.  The thicker line represents the $(4,4)$ entry of $K$.  The thinner line is the same curve as plotted in Figure \ref{unstabplot} with the value at $m = 0.4$ emphasized.
}
\label{stabplot} 
\end{center}
\end{figure}

We note first that if we restrict to the subspace $\mathcal{A}$, then the eigenvalue of $K$ given by the $(4,4)$ entry corresponds to linear stability of the 2DF orbit.  This value stays in the interval $[-1,1]$ for $m \in (0.01 + \epsilon, 1]$, giving the linear stability result claimed earlier. \\

Based on the additional results of the numerical calculations, we conclude that the 4DF rhomboidal orbit is linearly unstable, hence unstable, for all $m$ except for a small interval about $m = 0.4$.  Additionally, there are three values of $m$ for which we establish only spectral stability, due to repeated eigenvalues on the unit circle.  Roberts' argument (see \cite{bibRoberts1}) demonstrates that each of the computed eigenvalues of $K$ in $[-1,1]$ correspond to the real part of a square root of an eigenvalue on the complex unit circle.  Accordingly, the value of $m = m_1$ where the two curves in Figure \ref{stabplot} cross is a point with duplicated eigenvalues, hence only spectral stability.  Similarly, the value $m = m_0$ where the curve in Figure \ref{unstabplot} crosses the $x$-axis gives spectral stability, as $(\pm i)^2 = 1$.  A third point occurs where $\cos(2\alpha(m)) = \cos(2\beta(m))$, where $\alpha(m)$ and $\beta(m)$ are the the two curves plotted in Figure \ref{stabplot}. \\

In order to obtain a more precise intervals of mass values for which the orbit is linearly stable (excluding the above-mentioned $m_i$), the initial conditions for mass values $m = 0.39$, $0.391$, ..., $0.409$, $0.41$ were obtained using the same trigonometric polynomial approximation/optimization as used in Section \ref{Numerics}.  The same linear stability calculations demonstrate that the 4DF rhomboidal orbit is linearly stable for $m$ contained in some subinterval of $(0.395, 0.401)$.  In other words, the orbit was found to be linearly unstable for $m = 0.395$ and $m = 0.401$, but linearly stable for all computed values in between, with the exclusion of the three critical mass values.  \\

In \cite{bibBounemoura}, Bounemoura shows that in an $n$-dimensional Hamiltonian system, orbits beginning close to a linearly stable invariant torus generically remain ``close'' to the invariant torus for a super-exponential amount of time, eventually drifting away.  The theory can also be applied to other cases, such as elliptic fixed points of maps.  This analysis leads us to believe that, even though we have linear stability for some open interval containing $m = 0.4$, the orbit is likely still be unstable.  Numerical perturbations off of the invariant subspace $\mathcal{A}$ give evidence that this is the case.\\


\newpage

\bibliographystyle{spmpsci}
\bibliography{rhombCombNotes2}

\end{document}